\documentclass[final,leqno]{siamltex}

\usepackage{graphicx}
\usepackage{psfrag}
\usepackage{subfigure}
\usepackage{url}
\usepackage{cite}
\usepackage{epsfig}
\usepackage{amsmath,amssymb}

\newcommand{\rset}{\mathbb{R}}
\newcommand{\Gb}{\mathbf{G}}
\newcommand{\Eb}{\mathbf{E}}
\newcommand{\gb}{\mathbf{g}}
\newcommand{\Wb}{\mathbf{W}}
\newcommand{\wb}{\mathbf{w}}
\newcommand{\vb}{\mathbf{v}}
\newcommand{\bo}{\mathbf}
\newcommand{\bu}{\mathbf{u}}
\newcommand{\bv}{\mathbf{v}}
\newcommand{\eb}{\mathbf{e}}
\newcommand{\eo}{\epsilon_{\text{out}}}
\newcommand{\ei}{\epsilon_{\text{in}}}
\newcommand{\ec}{\epsilon_{\text{c}}}
\newcommand{\lu}{\mathcal{L}}
\newcommand{\qed}{\hfill \qedbox \\[1ex]}
\newtheorem{assumption}[theorem]{Assumption}
\newenvironment{remark}{
       \refstepcounter{theorem}\begin{trivlist}\item[]{\bfseries
      Remark \thetheorem\,}}
       {\end{trivlist}}

\newenvironment{algorithm}[1]
{\vskip0.1cm\noindent\textsc{Algorithm } $($#1$)$.}{}
\usecounter{algorithmenumi}

\newcommand{\iprod}[1]{\left\langle#1\right\rangle}
\newcommand{\set}[1]{\left\{#1\right\}}

\newcommand{\qedbox}{
\begin{minipage}[b]{3mm}
{\unitlength1pt
\begin{picture}(7,7)
\thicklines  \put(-0.2,0){\line(1,0){6.7}}
\put(6,0){\line(0,1){6.2}} \thinlines \put(0,0){\line(0,1){6}}
\put(0,6){\line(1,0){6}}
\end{picture}}
\end{minipage}}

\begin{document}


\title{Rate analysis of inexact dual first order methods: Application to distributed MPC for
 network systems}

\author{Ion Necoara and Valentin Nedelcu \thanks{The authors are with Automation and
 Systems Engineering Department,  University
Politehnica Bucharest, 060042 Bucharest, Romania. Corresponding
author: I.~Necoara, Tel. +40-21-4029195, Fax +40-21-4029195, Email
ion.necoara@acse.pub.ro.} }


\maketitle

\begin{abstract}
In this paper we propose and analyze two dual  methods based on
inexact  gradient information and averaging that generate
approximate primal solutions for smooth convex optimization
problems. The complicating constraints are moved into the cost
using the Lagrange multipliers. The dual problem is solved by
inexact first order methods based on approximate gradients and we
prove sublinear rate of convergence for these methods. In
particular, we provide, for the first time,  estimates on the
primal feasibility violation and primal and dual suboptimality of
the generated approximate primal and dual solutions. Moreover, we
solve approximately the inner problems with a parallel coordinate
descent algorithm  and we show that it has linear convergence
rate. In our analysis we rely on the Lipschitz property of the
dual function and inexact dual gradients. Further, we apply these
methods to distributed model predictive control for network
systems. By tightening the complicating constraints we are also
able to ensure the primal feasibility of the approximate solutions
generated by the proposed algorithms. We obtain a distributed
control strategy that has the following features: state and input
constraints are satisfied, stability of the plant is guaranteed,
whilst the number of iterations  for the suboptimal solution can
be precisely determined.
\end{abstract}


\pagestyle{myheadings} \thispagestyle{plain} \markboth{I. Necoara,
V. Nedelcu}{Rate analysis of inexact dual first order methods:
Application to distributed MPC}


\section{Introduction}
\label{sec_introduction} Different problems from control and
estimation can be addressed within the framework of network
systems \cite{NecNed:11}. In particular, model predictive control
(MPC) has become a popular advanced control technology implemented
in network systems due to its ability to handle hard input and
state constraints.  Network systems are  complex and large in
dimension, whose structure may be hierarchical, multistage or
dynamical and they have multiple decision-makers. Such systems can
be broken down into smaller, more malleable subsystems called
decompositions. How to consider the relationships between these
various decompositions has led to much of the recent work within
the general subject of the study of network systems.

Decomposition methods represent a  powerful tool for solving
distributed control, estimation and other engineering problems. The
basic idea of these methods is to decompose the original large
optimization problem into smaller subproblems which are then
coordinated by a master problem. Decomposition  methods can be
divided into two main classes: primal   and dual decomposition
 methods. In primal
decomposition  the optimization problem is solved using the original
formulation and variables, while the complicating constraints are
handled via methods  such as interior point, penalty functions,
feasible directions, Jacobi
\cite{CamSch:11,DunMur:06,FarSca:12,NecNed:11,SteVen:10}. In dual
decomposition the original problem is rewritten using Lagrangian
relaxation and then solve the dual problem
\cite{AlvLim:11,BerTsi:89,BoyPar:11,DoaKev:11,NecSuy:08}. When the
original problem is characterized by both simple and complicating
constraints, dual decomposition may represent an appropriate choice
since the complicating constraints can be moved into  the cost using
Lagrange multipliers and then  the inner problems, that have simple
constraints, are solved  and  the dual variables are updated  with a
Newton or (sub)gradient algorithm. Dual fast gradient methods based
on exact first order information with provable guarantees on
suboptimality are given in \cite{NecSuy:08} for general convex
problems and  \cite{PatBem:12} for QP's.  Dual methods based on
subgradient iteration and averaging, that produce primal solutions
in the limit, can be found e.g. in
\cite{KiwLar:07,LarPat:98,SenShe:86}. Converge rate analysis for the
dual subgradient method  has been studied e.g. in \cite{NedOzd:09},
where the authors provide estimates of order $\mathcal{O}
(1/\sqrt{k})$ for suboptimality and  feasibility violation of the
approximate solutions. Thus, an important drawback of the dual
methods is that feasibility of the primal variables can be ensured
only at optimality, which is usually impossible to attain in
practice. However, in many applications, e.g. from control and
estimation, the constraints can represent different requirements on
physical limitation of actuators, safety limits and operating
conditions of the controlled plant. Neglecting these constraints can
reduce economic profit and cause damage to the environment or
equipments. Therefore, any control or estimation scheme must ensure
feasibility. Further, there is no convergence rate analysis in any
of the existing literature  for inexact dual (fast) gradient
schemes. Thus, our goal is to develop  inexact dual gradient
algorithms which provide approximate primal solutions that are
suboptimal  and close to feasibility.

There are many ways to ensure feasibility of the primal variables in
distributed MPC, e.g. through  constraint tightening
\cite{CamSch:11,DoaKev:11,KuwRic:07,RicHow:07} or  distributed
implementations of some classical methods such as the method  of
feasible directions, penalty functions, Jacobi and others
\cite{CamOli:09, DunMur:06,FarSca:12,MaeMun:11,SteVen:10}. In
\cite{DoaKev:11},  a dual distributed algorithm for solving the MPC
problem for systems with coupled dynamics and constraints is
presented. The algorithm generates a primal feasible solution using
primal averaging and constraint tightening. The Jacobi algorithm
from \cite{BerTsi:89} is used to update the primal variables, while
the dual variables are updated using the subgradient method in
 \cite{NedOzd:09}. The authors prove the convergence of the algorithm
using the analysis of the dual subgradient method from
\cite{NedOzd:09} which has very slow convergence rate. In
\cite{KuwRic:07},  the authors propose a decentralized MPC algorithm
that uses the constraint tightening technique to achieve robustness
while guaranteeing robust feasibility of the entire system. In
\cite{RicHow:07,FarSca:12},  distributed MPC algorithms for systems
with coupled constraints is discussed. The approach divides the
single large planning optimization into smaller subproblems, each
planning only for the controls of a particular subsystem. Relevant
plan data is communicated between subproblems to ensure that all
decisions satisfy the coupled constraints. In
\cite{NecCli:12,SteVen:10} cooperative based distributed MPC
algorithms are proposed that converge to the centralized solution.
In \cite{MaeMun:11} a distributed MPC algorithm is proposed based on
agent  negotiation.   In \cite{CamSch:11,CamOli:09} distributed
algorithms based on interior point or feasible directions are
proposed that also converge to the centralized solution and
guarantees primal feasibility.  An iterative distributed model
predictive control  of large-scale nonlinear systems subject to
asynchronous and delayed state feedback is discussed in
\cite{LiuChe:12}.  See also \cite{ChrSca:12,NecNed:11,Sca:09} for
recent surveys of distributed and hierarchical MPC methods. While
most of the work cited above focuses on a primal approach, our work
develops for the first time efficient dual methods that ensure
constraint feasibility, tackles more general problems and more
complex constraints and provides much better estimates on
suboptimality.

{\em Contribution.} The contributions of the paper are as follows:
\begin{enumerate}
\item We propose and analyze novel dual algorithms with low complexity and
fast rate of convergence that generate approximate primal solutions
 for large smooth convex  problems.

\item We  introduce a general framework for inexact first order information and then
propose two inexact gradient methods for solving the dual (outer)
problem:
\begin{itemize}
\item an inexact dual gradient method, with rate of convergence of
order $\mathcal{O} ( 1 / k)$.

\item an inexact dual fast gradient method, with convergence rate
 of order $\mathcal{O}( 1 / k^2)$.
\end{itemize}

\item For both methods we provide for the first time a complete
rate analysis and estimates on primal/dual suboptimality and
feasibility violation of the generated approximate solutions.

\item In our schemes we solve the inner problems only up to a
certain accuracy $\ei$ by means of a parallel coordinate descent
method for which we prove linear rate of convergence.

\item For convex optimization models arising from distributed
MPC problems, we adapt our algorithms using a tightening constraints
approach, such that  the convergence rates of the methods are
preserved but in addition we are also able to ensure the primal
feasibility.

\item To certify the complexity of the proposed methods, we apply
the new algorithms on several linear distributed MPC problems with
state and input constraints.
\end{enumerate}

{\em Paper outline.} The paper is organized as follows. In Section
\ref{sec_mainsec} we introduce the dual problem of our original
optimization problem formulated in Section
\ref{sec_formulation_MPC}. In Sections \ref{sec_dg} and
\ref{sec_dfg} we develop inexact dual gradient and  fast gradient
schemes for solving the outer problem and analyze their convergence
rates. In Section \ref{sec_coordinate} we propose a parallel
coordinate descent method for solving the inner problems and prove
its convergence rate. In Section \ref{sec_application_MPC} we first
show how the distributed MPC problem for a network system can be
recast in the form of our optimization model. Then, we combine the
new dual algorithms with constraint tightening in order to ensure
primal feasibility and stability. Finally, in Section
\ref{sec_numerical} we provide extensive simulations in order to
certify the efficiency of the newly developed algorithms.

{\em Notation:} We work in the space $\rset^n$ composed by column
vectors. For $\bu, \bv \in \rset^n$ we denote the standard Euclidean
inner product $\langle \bu,\bv \rangle= \sum_{i=1}^n \bu_i \bv_i$,
norm $\left \| \bu \right \|=\sqrt{\langle \bu, \bu \rangle}$ and
projection onto non-negative orthant $\rset^n_+$ as $\left[ \bu
\right]_+$. We use $\langle \cdot,\cdot \rangle$, $\left \| \cdot
\right \|$ and $\left[ \cdot \right]_+$ for spaces of different
dimension. For a real number $\alpha$, $\lfloor{\alpha}\rfloor$
denotes the largest integer which is less than or equal to $\alpha$.
For any $\varepsilon \in [0,1]$ we say that a quantity $\textsl{q}$
is of order $\mathcal{O}(\textsl{p}(\varepsilon))$ if there exists
$c>0$ such that $\textsl{q} \leq c \textsl{p}(\varepsilon)$.
Further, for a convex set $\bo{U}$, $\text{relint}(\bo{U})$ denotes
the relative interior and $D_\bo{U}$ its diameter $D_\bo{U}=\max
\limits_{\bu, \bv \in \bo{U}} \|\bu - \bv\|$. For a matrix $G \in
\rset^{p \times n}$, $\|G\|$ and $\|G\|_F$ denote the $2$-norm and
Frobenius norm, respectively.


\subsection{Problem formulation} \label{sec_formulation_MPC}


We are interested in solving the following large-scale smooth
convex optimization problem:
\begin{equation}\label{original_primal}
F^*= \min_{\mathbf{u} \in \bo{U}} \left\{ F(\mathbf{u}):
~~h(\mathbf{u}) \leq 0\right\},
\end{equation}
where $F:\rset^n \rightarrow \rset$ and  the components of
$h:\rset^n \rightarrow \rset^p$ are convex functions, and $\bo{U}
\subseteq \rset^n$ is a compact, convex set. Further, we assume that
$F$ and the components of $h$ are twice differentiable. We also
assume that the projection on the set defined by the complicating
constraints  (called also \textit{coupling constraints}) $h(\bo{u})
\leq 0$ is hard to compute, but the set $\bo{U}$ is simple, i.e. the
projection on this set can be computed very efficiently (e.g.
hyperbox, Euclidean ball, etc).

In this paper we consider the following assumptions:
\begin{assumption}
\label{as_strong_lipschitz} $(i)$ Function $F$ is
$\sigma_{\text{F}}$-strongly convex w.r.t. $\|\cdot\|$ (see
\cite[Definition 2.1.2]{Nes:04}).

$(ii)$ The Jacobian of $h$ is bounded on $\bo{U}$, i.e. there exists
a constant $c_\text{h} > 0$ such that:  \[ \| \nabla h(\bo{u}) \|_F
\leq c_\text{h} \;\; \forall \bo{u} \in \bo{U}. \]
\end{assumption}

\begin{assumption} \label{as_Slater} Slater condition holds for
 \eqref{original_primal}, i.e. exists  $\tilde{\bo{u}}
\in \text{relint}(\bo{U})$ with $h(\tilde{\bo{u}}) < 0$.
\end{assumption}

Note that as a consequence of Assumption \eqref{as_Slater}, we
have that strong duality holds for \eqref{original_primal}.



\section{Solving the dual problem using inexact first order methods}
\label{sec_mainsec}
 Our goal is to solve the optimization problem
\eqref{original_primal} using dual gradient based methods. In order
to update the dual variables we use inexact dual gradient methods
(Sections \ref{sec_dg} and \ref{sec_dfg}), while the inner problems
are solved up to a certain accuracy by means of a parallel
coordinate descent algorithm (Section \ref{sec_coordinate}). An
important feature of our algorithms consists of the fact that even
if we use the inexact  gradient of the dual function, after a
certain number $k_\text{out}$ of outer iterations, we are still able
to compute a sequence of primal variables
$\bo{\hat{u}}^{k_\text{out}}$ which are $\eo$-optimal and their
feasibility violation is also less than $\mathcal{O}(\eo)$, i.e.:
\begin{equation}
\label{condition_outer} \bo{\hat{u}}^{k_\text{out}} \!\in \bo{U},
\; \|[h(\bo{\hat{u}}^{k_\text{out}})]^+\| \leq \mathcal{O} (\eo)
\;\; \text{and} \;\; -\mathcal{O}(\eo)\leq
F(\bo{\hat{u}}^{k_\text{out}}) - F^* \leq \mathcal{O} (\eo).
\end{equation}


\subsection{A framework for inexact first order information}
We assume that the projection on $\bo{U}$ is simple but the
projection on the set defined by the coupling constraints $h(\bo{u})
\leq 0$ is hard to compute. Therefore, we move the complicating
constraints into  the cost via Lagrange multipliers and define the
dual function:
\begin{align}
\label{dual_fc} d(\lambda) = \min_{\mathbf{u} \in
\bo{U}}\mathcal{L}(\mathbf{u},\lambda),
\end{align}
where $\mathcal{L}(\mathbf{u},\lambda)=F(\mathbf{u}) + \langle
\lambda, h(\mathbf{u}) \rangle$ denotes the partial Lagrangian
w.r.t. the complicating constraints $h(\bo{u})\leq 0$. We also
denote by $\bo{u}(\lambda)$ an optimal solution of the \textit{inner
problem}:
\begin{align}
\label{ul} \bo{u}(\lambda)  \in \arg \min_{\mathbf{u} \in \bo{U}}
\mathcal{L}(\mathbf{u},\lambda).
\end{align}
Based on Assumption \ref{as_strong_lipschitz} the gradient of the
dual function $d(\lambda)$ is given by \cite[Appendix A]{BerTsi:89}:
\[ \nabla  d(\lambda) = h(\bo{u}(\lambda)). \]
The following lemma gives a characterization of the Lipschitz
property for  the gradient $\nabla d (\lambda)$:

\begin{lemma}[see Appendix]
\label{technical_Lipschitz} Let the function $F$ and the
components of $h$ be twice differentiable and Assumption
\ref{as_strong_lipschitz} holds. Then, the gradient $\nabla
d(\lambda)$ is Lipschitz continuous with constant:
\begin{equation*}
L_{\text{d}}=\frac{c_\text{h}^2}{\sigma_{\text{F}}}.
\end{equation*}
\end{lemma}

\noindent Under strong duality (see Assumption \ref{as_Slater}) we
have for the \textit{outer problem}:
\begin{align}
\label{dual_pr} F^* = \max_{\lambda \geq 0} d(\lambda),
\end{align}
for which we denote an optimal solution by $\lambda^*$. Since we
cannot usually solve the inner optimization problem \eqref{ul}
exactly, but with some inner accuracy  obtaining an approximate
optimal solution $\bo{\bar u}(\lambda)$, we have to use inexact
gradients and approximate values of the dual function $d$. Thus, we
introduce the following two notions:
\[ \bar{d}(\lambda) = \mathcal{L}(\bo{\bar u}(\lambda),\lambda) ~\mathrm{and}~{\bar \nabla}  d(\lambda) = h(\bo{\bar u}(\lambda)).  \]
If we assume that $\bo{\bar u}(\lambda)$ is computed such that the
following inner $\ei$-optimality holds:
\begin{equation}
\label{inner_crit} \mathbf{ \bar u}(\lambda) \in \bo{U}, \;\;
\mathcal{L}(\mathbf{ \bar
u}(\lambda),\lambda)-\mathcal{L}(\mathbf{u}(\lambda),\lambda) \leq
\frac{\ei}{3},
\end{equation}
 then the
next lemma provides  bounds for the dual function $d(\lambda)$ in
terms of a linear and  a quadratic model which use only approximate
information of the dual function and of its gradient.
\begin{lemma}\cite[Section 3.2]{DevGli:11} Let Assumptions \ref{as_strong_lipschitz} and
\ref{as_Slater} hold and for a given $\lambda$
let $\mathbf{\bar u}(\lambda)$ be computed such that
\eqref{inner_crit} is satisfied. Then, the following inequalities
are valid:
\begin{align}
\label{ineq_approx} 0 \geq d(\mu)-[ \bar{d}(\lambda)+ \langle {
\bar \nabla} d(\lambda),\mu -\lambda \rangle ] \geq - L_{\text{d}}
\|\mu-\lambda\|^2-\ei \quad \forall \mu \in \rset^p_{+}.
\end{align}
\end{lemma}
\begin{proof}
For linear functions $h$,  this lemma is proved in \cite[Section
3.2]{DevGli:11} with stopping criterion $\ei/2$ in
\eqref{inner_crit}. For general convex functions $h$ satisfying
Assumption \ref{as_strong_lipschitz} $(ii)$ we can easily show that
$\| h(\bu) - h(\bv) \| \leq \sqrt{2} c_\text{h} \| \bu - \bv\| $ and
then following exactlty the same steps as in \cite{DevGli:11} we get
the result in \eqref{ineq_approx}.
\end{proof}

\begin{remark} Relation \eqref{inner_crit} represents
the stopping criterion for solving the inner problem \eqref{ul}.
Many optimization methods offer direct control of this criterion
(see e.g. the method of Section \ref{sec_coordinate}). For affine
functions $h$ (see e.g. MPC problems in Section
\ref{sec_application_MPC}),  the stopping criterion in
\eqref{inner_crit} can be taken as \cite{DevGli:11}: $\mathbf{ \bar
u}(\lambda) \in \bo{U}$ and $\mathcal{L}(\mathbf{ \bar
u}(\lambda),\lambda)-\mathcal{L}(\mathbf{u}(\lambda),\lambda) \leq
\frac{\ei}{2}$.

\end{remark}


\subsection{Inexact dual gradient method for solving the dual
(outer) problem} \label{sec_dg}

In this section we analyze the convergence properties of an inexact
dual projected gradient algorithm for solving approximately  the
dual problem \eqref{dual_pr}. Let $\set{\alpha^j}_{j\geq 0}$ be a
sequence of positive numbers and $S^k = \sum_{j=0}^k \alpha^j$. We
consider the following inexact dual gradient algorithm:

\begin{center}
\framebox{
\parbox{9.5cm}{
\begin{center}
\textbf{ Algorithm {\bf (IDG)}$(\lambda^0)$ }
\end{center}
{Given $\lambda^0 \in \rset^p_+$, for $k\geq 0$ compute:}
\begin{enumerate}
\item{$\mathbf{ \bar u}^k \approx \arg \min\limits_{\mathbf{u} \in
\bo{U}} \mathcal{L}(\mathbf{u},\lambda^k)$ { such that
\eqref{inner_crit} holds}} \item
$\lambda^{k+1}=\left[\lambda^k+\alpha^k{\bar \nabla}
d(\lambda^k)\right]_+$.
\end{enumerate}
}}
\end{center}

Recall that  inexact gradient ${\bar \nabla} d(\lambda^k)=h(\bo{
\bar u}^k)$ and $\alpha^k \in
\left[\frac{1}{2\underline{L}},\frac{1}{2L_{\mathrm{d}}}\right]$
is a given step size with~$\underline{L} \geq L_{\mathrm{d}}$. The
following theorem provides an estimate on the dual suboptimality
for algorithm {\bf (IDG)}:
\begin{theorem}
\label{dual_optimality_grad} Let Assumptions
\ref{as_strong_lipschitz} and \ref{as_Slater} hold and the sequences
$\left(\mathbf{ \bar u}^k,\lambda^k \right)_{k\geq 0}$ be generated
by algorithm {\bf (IDG)} and define the average sequence of  dual
variables $\hat{\lambda}^k =\frac{1}{S^k}\sum_{j=0}^{k} \alpha^j
\lambda^{j+1}$. Then, the following estimate on dual suboptimality
can be derived for  dual problem \eqref{dual_pr}:
\begin{equation}
\label{eq_conv_dual_grad} F^* - d(\hat{\lambda}^k) \leq
\frac{\underline{L} R_{\mathrm{d}}^2}{k+1}+\epsilon_{\text{in}},
\end{equation}
where we define:
\begin{equation*}
R_{\mathrm{d}}=\|\lambda^*-\lambda^0\|.
\end{equation*}
\end{theorem}
\proof Let us first notice that the update of the dual variables
can be equivalently written as $\lambda^{k+1}=\arg \min
\limits_{\lambda \geq 0}{\left[
\frac{1}{2\alpha^k}\left\|\lambda-\lambda^k\right\|^2-\iprod{{\bar
\nabla} d(\lambda^k),\lambda-\lambda^k}\right]}$, for which the
optimality condition reads:
\begin{equation}
\label{opt_cond_iter_grad}
\iprod{\lambda^{k+1}-\lambda^k-\alpha^k{\bar \nabla}
d(\lambda^k),\lambda - \lambda^{k+1}} \geq 0 ~~~ \forall \lambda
\geq 0.
\end{equation}
If we now define $r^{j}_{\lambda}=\|\lambda^j-\lambda\|^2$ for any
$\lambda \geq 0$, then we have:
\begin{align}
\label{inequalities_gradient} r^{j+1}_{\lambda}
&=\|\lambda^{j+1}-\lambda^j+\lambda^j-\lambda\|^2=
r^{j}_{\lambda}+2\langle
\lambda^{j+1}-\lambda^j,\lambda^{j}-\lambda^{j+1}+\lambda^{j+1}-\lambda
\rangle+\|\lambda^{j+1}-\lambda^j\|^2 \nonumber\\
&=r^{j}_{\lambda}+2\langle
\lambda^{j+1}-\lambda^j,\lambda^{j+1}-\lambda
\rangle-\|\lambda^{j+1}-\lambda^j\|^2 \nonumber \\
&\overset{\eqref{opt_cond_iter_grad}}{\leq}
r^{j}_{\lambda}-2\alpha^j\langle \bar{\nabla}
d(\lambda^j),\lambda-\lambda^j\rangle +2\alpha^j\left[\langle
\bar{\nabla}
d(\lambda^j),\lambda^{j+1}\!\!-\!\!\lambda^j\rangle\!-
\!L_{\text{d}}\|\lambda^{j+1}\!\!-\!\!\lambda^j\|^2\right]  \\
& \overset{\eqref{ineq_approx}}{\leq}
r^{j}_{\lambda}+2\alpha^j\left[\bar{d}(\lambda^j)-d(\lambda)\right]+2\alpha^j\left[d(\lambda^{j+1})-\bar{d}(\lambda^j)+\ei\right]
\nonumber \\
&=
r^{j}_{\lambda}+2\alpha^j[d(\lambda^{j+1})-d(\lambda)+\epsilon_{\text{in}}]
~~ \forall \lambda \geq 0, \nonumber
\end{align}
where in the first inequality we use the fact that $\alpha^j \leq
\frac{1}{2L_{\mathrm{d}}}$. Summing up these inequalities for
$j=0,\dots,k$ and using the definition of $\hat{\lambda}^k$ we can
write:
\begin{equation*}
2S^k \left[d(\lambda)-d(\hat{\lambda}^k)\right] \leq
r_{\lambda}^0+2S^k\ei ~~~ \forall \lambda \geq 0.
\end{equation*}
Letting now $\lambda = \lambda^*$, dividing both sides of the
previous inequality by $2S^k$ and taking into account that $S^k
\geq \frac{k+1}{2\underline{L}}$ we obtain
\eqref{eq_conv_dual_grad}. \qed

\noindent We can observe that the first term in the estimate
\eqref{eq_conv_dual_grad} represents the standard rate of
convergence  of the gradient method for the class of smooth
functions \cite{Nes:04}. Also, the second term $\ei$  is the error
induced by the fact that the gradient is computed only approximately
and shows that algorithm {\bf (IDG)} does not accumulate errors.

However, we are now interested in finding estimates for primal
feasibility violation for original problem \eqref{original_primal}.
Let us introduce  the following average primal sequence:
\begin{equation} \label{avg} \mathbf{ \hat
u}^k=\frac{1}{S^k}\sum_{j=0}^k \alpha^j \mathbf{\bar{u}}^j.
\end{equation}
The following  theorem provides  an estimate on primal feasibility
violation for problem \eqref{original_primal}:
\begin{theorem}
\label{theorem_fesa_grad} Under the assumptions of Theorem
\ref{dual_optimality_grad} and with $\mathbf{ \hat u}^k$ defined in
\eqref{avg}, the following estimate on primal feasibility violation
can be derived for the original problem \eqref{original_primal}:
\begin{equation}
\label{eq_fesa_grad} \|[h(\mathbf{ \hat u}^k)]_+\| \leq
v(k,\epsilon_{\text{in}}) \;\; \forall k \geq 0,
\end{equation}
where $ v(k,
\epsilon_{\text{in}})=\frac{4\underline{L}R_{\mathrm{d}}}{k+1}+\frac{6\underline{L}\|\lambda^0\|}{k+1}+2\sqrt{\frac{\underline{L}}{k+1}\ei}.$
\end{theorem}
\proof Using the definition of $\lambda^{j+1}$ we have that the
following component-wise inequalities hold: $\lambda^j+\alpha^j
\bar{\nabla} d(\lambda^j) \leq \lambda^{j+1}$ for all $j \geq 0$.
Summing up these inequalities for $j=0,\dots,k$ and taking into
account that $\bar{\nabla} d(\lambda^j)=h(\mathbf{ \bar u}^j)$ we
obtain: $\sum_{j=0}^k  \alpha^j h(\mathbf{ \bar u}^j) \leq
\lambda^{k+1}-\lambda^0 \leq \lambda^{k+1}$,  which together with
the convexity of $h$ gives: $h(\mathbf{ \hat u}^k) \leq
\frac{\lambda^{k+1}}{S^k}$. Since $\lambda^{k+1} \geq 0$ we also
have that $0 \leq \left[h(\mathbf{ \hat u}^k) \right]_+ \leq
\frac{\lambda^{k+1}}{S^k}$ and thus we can  further  write:
\begin{equation}
\label{fesability+multipliers} \|\left[h(\mathbf{ \hat u}^k)
\right]_+ \| \leq \frac{\|\lambda^{k+1}\|}{S^k}.
\end{equation}
Thus, in order to find an estimate on primal feasibility violation,
we have to upper bound the norm of the dual sequence
$\lambda^{k+1}$. For this purpose we can use
\eqref{inequalities_gradient} with $\lambda=\lambda^*$:
\begin{align*}
\|\lambda^{j+1} - \lambda^*\|^2  \leq \|\lambda^j - \lambda^*\|^2
+ 2\alpha^j \left[ d(\lambda^{j+1}) - d(\lambda^*) +
\epsilon_{\text{in}} \right].
\end{align*}
Summing up these inequalities for $j=0, \dots, k$, using
$\iprod{\lambda^0,\lambda^*} \geq 0$ and $d(\lambda^{j+1}) \leq
d(\lambda^*)$, we~get:
\begin{equation*}
\|\lambda^{k+1}-\lambda^*\|^2  \leq \|\lambda^*\|^2
+\|\lambda^0\|^2 +2S^k\epsilon_{\text{in}},
\end{equation*}
Now, using the Cauchy-Schwartz inequality we get the  second order
inequality in $\|\lambda^{k+1}\|$:
\begin{equation*}
\|\lambda^{k+1}\|^2-2\|
\lambda^*\|\|\lambda^{k+1}\|-\|\lambda^0\|^2-2S^k
\epsilon_{\text{in}} \leq 0.
\end{equation*}
Therefore, $\|\lambda^{k+1}\|$ must be less than the largest root
of the corresponding second-order equation:
\begin{align*}
\label{eq_bound_mult_grad} \|\lambda^{k+1}\| &\leq
\frac{2\|\lambda^*\|+\left[4\|\lambda^*\|^2+4\|\lambda^0\|^2+8S^k\ei\right]^{1/2}}{2}\leq
2\|\lambda^*\|+\|\lambda^0\|+\sqrt{2S^k\ei}\\
&\leq 2\|\lambda^*-\lambda^0\|+3\|\lambda^0\|+\sqrt{2S^k\ei},
\end{align*}
where in the second inequality we used that $\sqrt{\zeta_1+\zeta_2}
\leq \sqrt{\zeta_1}+\sqrt{\zeta_2}$. Introducing this inequality in
\eqref{fesability+multipliers} and taking into account that $S^k
\geq \frac{k+1}{2\underline{L}}$ we obtain \eqref{eq_fesa_grad}.\qed

\begin{theorem}
\label{theorem_primal_grad} Let the assumptions of Theorem
\ref{theorem_fesa_grad} hold. Then, the following estimates on
primal suboptimality can be derived for the original problem
\eqref{original_primal}:
\begin{equation}
\label{eq_primal_grad}
-\left(R_{\mathrm{d}}+\|\lambda^0\|\right)v(k,\ei)\leq F(\mathbf{
\hat u}^k) - F^* \leq \frac{
\underline{L}\|\lambda^0\|^2}{k+1}+\ei.
\end{equation}
\end{theorem}
\proof In order to prove the left-hand side inequality we can
write:
\begin{align*}
F^*&=d(\lambda^*)=\min_{\mathbf{u}\in\mathbf{U}}
F(\mathbf{u})+\langle \lambda^*,h(\mathbf{u})\rangle \leq
F(\mathbf{ \hat u}^k)+\langle \lambda^*,h(\mathbf{ \hat
u}^k)\rangle \\
&\leq F(\mathbf{ \hat u}^k)+ \iprod{\lambda^*,[h(\mathbf{ \hat
u}^k)]_+}  \leq F(\mathbf{ \hat u}^k) + \|\lambda^*\|\|[h(\mathbf{
\hat u}^k)]_+\|\\
&= F(\mathbf{ \hat u}^k) +
\|\lambda^*-\lambda^0+\lambda^0\|\|[h(\mathbf{ \hat u}^k)]_+\|
\nonumber \leq F(\mathbf{ \hat u}^k) +
(R_{\mathrm{d}}+\|\lambda^0\|)\|[h(\mathbf{ \hat u}^k)]_+\|,
\end{align*}
which together with \eqref{eq_fesa_grad} lead to the result.

\noindent Now, we prove the right-hand side inequality. Taking
$\lambda=0$ in the first inequality of \eqref{inequalities_gradient}
we~get:
\begin{align*}
\|\lambda^{j+1}\|^2-2\alpha^j\langle \bar{\nabla}
d(\lambda^j),\lambda^j\rangle &\leq \|\lambda^j\|^2
+2\alpha^j\left[\langle \bar{\nabla}
d(\lambda^j),\lambda^{j+1}\!\!-\!\!\lambda^j\rangle\!-
\!L_{\text{d}}\|\lambda^{j+1}\!\!-\!\!\lambda^j\|^2\right]\\
&\overset{\eqref{ineq_approx}}{\leq}
\|\lambda^j\|^2+2\alpha^j\left[d(\lambda^{j+1})-\bar{d}(\lambda^j)+\ei\right].
\end{align*}
Taking into account that $\bar{\nabla} d(\lambda^j)=h(\bo{\bar
u}^j)$ and using the definition of $\bar{d}(\lambda^j)$ we have:
$-\langle \bar{\nabla} d(\lambda^j),\lambda^j\rangle=F(\bo{\bar
u}^j)-\bar{d}(\lambda^j)$.  Using this relation in the previous
inequality we get:
\begin{align*}
\|\lambda^{j+1}\|^2 +2\alpha^j \left[F(\bo{\bar
u}^j)-\bar{d}(\lambda^j)\right]\leq \|\lambda^j\|^2+2 \alpha^j
\left[d(\lambda^{j+1})-\bar{d}(\lambda^j)\right]+2\alpha^j\ei.
\end{align*}
Summing up these inequalities for $j=0, \dots, k$ and taking into
account that $F$ is convex and $d$ concave, we obtain the following
inequality:
\begin{equation*}
2S^k \left[F(\mathbf{ \hat u}^k)-d(\hat{\lambda}^k)\right]\leq
\|\lambda^0\|^2+2S^k\ei.
\end{equation*}
Dividing both sides of the previous inequality by $S^k$ and using
that $S^k \geq \frac{k+1}{2\underline{L}}$ and $d(\hat{\lambda}^k)
\leq F^*$, we obtain \eqref{eq_primal_grad}. \qed

Now, for a desired accuracy $\eo$ for solving problem
\eqref{original_primal}, we are interested in finding the number of
outer iterations $k_{\text{out}}$ and a relation between $\eo$ and
$\ei$ such that primal feasibility violation and suboptimality
satisfy  \eqref{condition_outer} and, moreover,  the  dual
suboptimality will be also less than $\mathcal{O}(\eo)$. For
simplicity, we consider the initial iterate $\lambda^0=0$ and thus
$R_{\mathrm{d}}=\|\lambda^*\|$. Further, we consider a constant step
size $\alpha^j=\frac{1}{2L_{\mathrm{d}}}$. Using Theorems
\ref{dual_optimality_grad}, \ref{theorem_fesa_grad} and
\ref{theorem_primal_grad} we can take:
\begin{equation*}
k_{\text{out}}=\left\lfloor\frac{4L_{\mathrm{d}}R_{\mathrm{d}}^2}{\eo}\right\rfloor~~\text{and}~~\ei=\eo,
\end{equation*}
for which we obtain the following estimates for primal feasibility
violation and    suboptimality:
\begin{align*}
&\bo{\hat{u}}^{k_{\text{out}}} \in \bo{U},\;
\|[h(\bo{\hat{u}}^{k_{\text{out}}})]_+\| \!\leq\!
\frac{2}{R_{\mathrm{d}}}\eo~,\\
-2\eo \leq F&(\bo{\hat{u}}^{k_{\text{out}}})-F^* \leq
\eo~~\mathrm{and}~~ F^*- d(\hat{\lambda}^{k_{\text{out}}}) \!\leq
\frac{5}{4}\eo.
\end{align*}
From the previous discussion it follows that in the algorithm
\textbf{(IDG)} the inner problems \eqref{ul} need to be solved
with about the same  accuracy as the desired accuracy  of the
outer problem, i.e. $\ei=\eo$ in the stopping criterion
\eqref{inner_crit}.


\subsection{Inexact dual fast gradient method for solving the
 dual (outer) problem} \label{sec_dfg}

In this section we discuss an inexact dual fast gradient scheme for
updating the dual variable $\lambda$. A similar algorithm was
proposed by Nesterov in \cite{Nes:04a} and applied further in
\cite{NecSuy:08} for solving dual problems with exact gradient
information. An inexact version of the algorithm can be also found
in \cite{DevGli:11}. The scheme defines two sequences $\left({\hat
\lambda}^k,\lambda^k\right)_{k\geq 0}$ for the dual variables:
\begin{center}
\framebox{
\parbox{10.5cm}{
\begin{center}
\textbf{ Algorithm {\bf (IDFG)}$(\lambda^0)$ }
\end{center}
{Given $\lambda^0 \in \rset^p_+$, for $k\geq 0$ compute:}
\begin{enumerate}
\item{$\mathbf{ \bar u}^k \approx \arg \min\limits_{\mathbf{u} \in
\bo{U}} \mathcal{L}(\mathbf{u},\lambda^k)$ { such that
\eqref{inner_crit} holds}}

\item ${\hat
\lambda}^k=\left[\lambda^k+\frac{1}{2L_{\text{d}}}{\bar \nabla}
d(\lambda^k)\right]_+$

\item $\lambda^{k+1}=\frac{k+1}{k+3}{\hat
\lambda}^k+\frac{2}{k+3}\left[\lambda^0+\frac{1}{2L_{\text{d}}}\sum_{s=0}^k
\frac{s+1}{2} {\bar \nabla} d(\lambda^s)\right]_+$,
\end{enumerate}
}}
\end{center}

\vspace{0.1cm} where we recall that ${\bar \nabla}
d(\lambda^k)=h(\bo{\bar u}^k)$. Based on Theorem 4 in
\cite{DevGli:11}, which is an extension of the results in
\cite{NecSuy:08,Nes:04a} to the inexact case, we have the
following result which will help us to establish upper bounds on
primal and dual suboptimality and feasibility violation for our
method.
\begin{lemma}\cite[Theorem 4]{DevGli:11}
\label{prop_rec} If Assumptions \ref{as_strong_lipschitz} and
\ref{as_Slater} hold and the sequences $\left(\mathbf{ \bar
u}^k,{\hat \lambda}^k,\lambda^k\right)_{k\geq 0}$ are generated by
algorithm {\bf (IDFG)}, then for all $k \geq 0$ we have:
\begin{align}
\label{eq_rec} \frac{(k+1)(k+2)}{4} d({\hat \lambda}^k) &\geq
\max_{\lambda \geq 0} - L_{\text{d}} \|\lambda-\lambda^0\|^2
+\sum_{s=0}^k \frac{s+1}{2}\left[\bar{d}(\lambda^s)+\langle {\bar
\nabla} d(\lambda^s),\lambda-\lambda^s
\rangle\right] \nonumber \\
&~~~~-\frac{(k+1)(k+2)(k+3)}{12}\epsilon_{\text{in}} ~~~\forall
\lambda \in \rset^p_+.
\end{align}
\end{lemma}

The following theorem provides an estimate on the dual suboptimality
for algorithm {\bf (IDFG)}:
\begin{theorem}
\label{theorem_dual_optim} Let Assumptions \ref{as_strong_lipschitz}
and \ref{as_Slater} hold and the sequences $\left(\mathbf{ \bar
u}^k,{\hat \lambda}^k,\lambda^k\right)_{k\geq 0}$ be generated by
algorithm {\bf (IDFG)}. Then, an estimate on dual suboptimality for
\eqref{dual_pr} is given by:
\begin{equation}
\label{bound_dual_optim} F^*-d({\hat \lambda}^k)\leq \frac{4
L_{\text{d}}R_{\mathrm{d}}^2}{(k+1)^2}+(k+1)\ei,
\end{equation}
with $R_{\mathrm{d}}$ defined as in Theorem
\ref{dual_optimality_grad}.
\end{theorem}

\proof Using the first inequality from \eqref{ineq_approx} in
\eqref{eq_rec} we get:
\begin{align*}
\frac{(k+1)(k+2)}{4} d({\hat \lambda}^k) \geq &- L_{\text{d}}
\|\lambda^0-\lambda^*\|^2+\sum_{s=0}^k\frac{s+1}{2}d(\lambda^*)-\frac{(k+1)(k+2)(k+3)}{12}\ei.
\end{align*}
Dividing now both sides  by $\frac{(k+1)(k+2)}{4}$, rearranging the
terms and taking into account that $d(\lambda^*)=F^*$, $(k+1)^2 \leq
(k+1)(k+2)$ and $(k+3)/3 \leq k+1$ we obtain
\eqref{bound_dual_optim}. \qed

\noindent We can observe that the first term in the estimate
\eqref{bound_dual_optim} represents the standard rate of convergence
of the fast gradient method  for the class of smooth functions
\cite{Nes:04}. Also, the second term $(k+1)\ei$
 is the error induced by the fact that the gradient is
computed only approximately and shows that algorithm {\bf (IDFG)}
accumulates the errors.

\noindent Further, we are interested now in finding estimates on
primal feasibility violation and primal suboptimality for our
original problem \eqref{original_primal}. For this purpose we define
the following average  sequence for the primal variables:
\begin{equation}
\label{primal_point} \mathbf{\hat
u}^k=\sum_{s=0}^k\frac{2(s+1)}{(k+1)(k+2)}\mathbf{\bar u}^s.
\end{equation}

The next result gives an estimate on primal feasibility violation.
\begin{theorem}
\label{theorem_primal_fesa} Under the assumptions of Theorem
\ref{theorem_dual_optim} and $\mathbf{\hat u}^k$ generated by
\eqref{primal_point}, an estimate on primal feasibility violation
for original problem \eqref{original_primal} is given by:
\begin{equation}
\label{ineq_bound_fesa_fg} \|[ h(\bo{\hat u}^k)]_+\| \leq
v(k,\epsilon_{\textrm{in}}),
\end{equation}
where $v(k,\ei)=\frac{16L_{\text{d}}
R_{\mathrm{d}}}{(k+1)^2}+\frac{8L_{\text{d}}
\|\lambda^0\|}{(k+1)^2}+4\sqrt{\frac{L_{\mathrm{d}}}{k+1}\ei}$.
\end{theorem}
\proof Using \eqref{eq_rec}, the convexity of $F$ and $h$ and taking
into account that $(k+3)/3 \leq k+1$, we can write for any $\lambda
\in \rset_+^p$:
\begin{equation}
\label{ineq_fesa0} \max_{\lambda \geq
0}-\frac{4L_{\text{d}}}{(k+1)^2}\|\lambda-\lambda^0\|^2+\langle
\lambda, h(\mathbf{ \hat u}^k) \rangle \leq (k+1)\ei + d({\hat
\lambda}^k) - F(\mathbf{\hat u}^k).
\end{equation}

For the second term of the right-hand side we have:
\begin{align}
\label{ineq_fesa1}  d({\hat \lambda}^k) -  F(\mathbf{ \hat u}^k)
&\leq d(\lambda^*)-F(\mathbf{ \hat u}^k) = \min_{\mathbf{u} \in
\bo{U}}F(\mathbf{u}) +\langle \lambda^*, h(\mathbf{u}) \rangle
-F(\mathbf{\hat u}^k) \nonumber\\
&\leq F(\mathbf{ \hat u}^k) + \langle \lambda^*, h(\mathbf{ \hat
u}^k)\rangle - F(\mathbf{ \hat u}^k)=\langle \lambda^*, h(\mathbf{
\hat u}^k) \rangle  \leq  \langle\lambda^*,[h(\mathbf{ \hat u}^k)
]_+\rangle,
\end{align}
where in the last inequality we used that $\lambda^* \geq 0$. By
evaluating the left-hand side term in \eqref{ineq_fesa0} at
$\lambda=\frac{(k+1)^2}{8L_{\mathrm{d}}}\left[ h(\mathbf{ \hat
u}^k) \right]_+$ and taking into account that $\langle
[h(\!\mathbf{ \hat u}^k\!) ]_+,h(\!\mathbf{ \hat u}^k\!)
-[h(\!\mathbf{ \hat u}^k\!) ]_+\rangle=0$ we obtain the following
inequality:
\begin{align}
\label{ineq_fesa2} \max_{\lambda \geq
0}-\frac{4L_{\text{d}}}{(k+1)^2}\|\lambda-\lambda^0\|^2+\langle
\lambda, h(\mathbf{ \hat u}^k) \rangle \geq
&\frac{(k+1)^2}{16L_{\text{d}}}\|[h(\mathbf{ \hat u}^k) ]_+\|^2
\nonumber\\
&-\frac{4L_{\mathrm{d}}\|\lambda^0\|^2}{(k+1)^2}+\langle
\lambda^0,[h(\mathbf{ \hat u}^k) ]_+\rangle.
\end{align}
Combining now \eqref{ineq_fesa1} and \eqref{ineq_fesa2} with
\eqref{ineq_fesa0}, using the Cauchy-Schwartz inequality and
introducing the notation $\alpha=\|\left[h(\mathbf{ \hat
u}^k)\right]_+\|$, we obtain the following second order inequality
in $\alpha$:
\begin{equation*}
\frac{(k+1)^2}{16L_{\text{d}}}\alpha^2-\|\lambda^*-\lambda^0\|\alpha-(k+1)\ei-\frac{4L_{\mathrm{d}}\|\lambda^0\|^2}{(k+1)^2}
\leq 0.
\end{equation*}
Therefore, $\alpha$ must be less than the largest root of the
second-order equation, from which together with the definition of
$R_{\mathrm{d}}$ and the identity $\sqrt{\zeta_1+\zeta_2} \leq
\sqrt{\zeta_1}+\sqrt{\zeta_2}$, we get the result. \qed

\begin{theorem}
\label{theorem_primal_optim} Assume that the conditions in Theorem
\ref{theorem_primal_fesa} are satisfied and let $\mathbf{\hat u}^k$
be given by \eqref{primal_point}. Then, the following estimate on
primal suboptimality  for  problem \eqref{original_primal} can be
derived:
\begin{equation}
\label{bound_primal_optim}
-\left(R_{\mathrm{d}}+\|\lambda^0\|\right)v(k,\ei)\leq F(\mathbf{
\hat u}^k) - F^* \leq \frac{4
L_{\text{d}}\|\lambda^0\|^2}{(k+1)^2}+(k+1  )\ei.
\end{equation}
\end{theorem}
\proof The left-hand side inequality can be derived similarly as in
 the previous section (see the proof of Theorem
\ref{theorem_primal_grad}). In order to prove the right-hand side
inequality we use \eqref{ineq_fesa0}:
\begin{align*}
F(\mathbf{ \hat u}^k) \!-\! d({\hat \lambda}^k) & \!\leq\!
 - \max_{\lambda \geq 0 } - \frac{4 L_{\text{d}}}{(k+1)^2} \| \lambda -\lambda^0\|^2
+ \langle \lambda, h(\mathbf{ \hat u}^k) \rangle
+\frac{k+3}{3}\ei\\
&\!\!\!\overset{\lambda=0}{\leq} \frac{4
L_{\text{d}}\|\lambda^0\|^2}{(k+1)^2}+\frac{k+3}{3}\ei.
\end{align*}
Taking now into account that $d({\hat \lambda}^k) \leq F^*$ and
$(k+3)/3 \leq k+1$ we get the result. \qed

Similar to the previous section, assume that we fix the outer
accuracy to a desired value $\eo$. We are interested in finding the
number of outer iterations $k_{\text{out}}$ and a relation between
$\eo$ and $\ei$ such that primal feasibility violation and
suboptimality satisfy \eqref{condition_outer}.  For simplicity, we
again  consider  $\lambda^0=0$ and thus
$R_{\mathrm{d}}=\|\lambda^*\|$. Using now Theorems
\ref{theorem_dual_optim}, \ref{theorem_primal_fesa} and
\ref{theorem_primal_optim} we can take:
\begin{equation*}
k_{\text{out}}=\left\lfloor 2R_{\mathrm{d}}\sqrt{\frac{
L_{\text{d}}}{\eo}}\right\rfloor~~\text{and}~~\ei=\frac{\eo\sqrt{\eo}}{2R_{\mathrm{d}}\sqrt{L_{\mathrm{d}}}},
\end{equation*}
for which we obtain:
\begin{align*}
&\bo{\hat{u}}^{k_{\text{out}}} \in \bo{U},\;
\|[h(\bo{\hat{u}}^{k_{\text{out}}})]_+\| \!\leq\!
\frac{6}{R_{\mathrm{d}}}\eo~,\\
 -6\eo \leq F&(\bo{\hat{u}}^{k_{\text{out}}})-F^* \leq
2\eo~~\mathrm{and}~~F^*- d(\hat{\lambda}^{k_{\text{out}}}) \!\leq
3\eo.
\end{align*}

Note that for these choices of $k_{\text{out}}$ and $\ei$ the inner
problems \eqref{ul} have to be solved with an accuracy of order
$\mathcal{O}\left(\eo\sqrt{\eo}\right)$, i.e.
$\ei=\eo\sqrt{\eo}/(2R_{\mathrm{d}} \sqrt{L_{\mathrm{d}}})$ in
\eqref{inner_crit}. We can conclude that the \textbf{(IDFG)} method
is more sensitive than the \textbf{(IDG)} method due to the error
accumulation.

\begin{remark} $(i)$ Since in practice we usually cannot compute
exactly the value $R_{\mathrm{d}}=\|\lambda^*\|$, we can use instead
the following upper bound \cite[Lemma 1]{NedOzd:09}:
\begin{equation}
\label{eq_upper_norm_mult} R_{\mathrm{d}} \leq
\mathcal{R}_{\mathrm{d}}=\frac{F(\tilde{\bu})-d(\tilde{\lambda})}{\min_{1\leq
j \leq p} \left\{-h^j(\tilde{\bu})\right\}},
\end{equation}
where $\tilde{\bu}$ denotes a Slater vector for  problem
\eqref{original_primal} (see Assumption \ref{as_Slater}) and
$\tilde{\lambda} \in \rset^p_+$. The effects of this choice  on the
overall
 performance of the new algorithms are discussed in Section \ref{test1}. \\
$(ii)$ The results presented in Sections \ref{sec_dg} and
\ref{sec_dfg} also hold in the case when we solve
 the inner problems exactly, i.e. $\ei = 0$ in \eqref{inner_crit}, or
when $\bo{U} = \rset^n$, i.e. the inner problems are
unconstrained. \\
$(iii)$ Note that if $\lambda^0=0$ and we solve the inner problems
exactly, i.e. $\ei=0$, then we have $F(\hat{\bu}^{k_{\mathrm{out}}})
\leq F^*$, i.e. we are always below the optimal value in  algorithms
\textbf{(IDG)} and \textbf{(IDFG)}.
\end{remark}


\section{Solving the inner problem using a parallel coordinate descent method}
\label{sec_coordinate}

In this section we propose a block-coordinate descent based
algorithm which permits to solve in parallel, for a fixed
$\lambda^k$, the inner optimization problem \eqref{ul}:
\begin{equation}
\label{inner_pcd} \bo{u}^k=\arg \min_{\bo{u} \in \bo{U}}
\mathcal{L}(\bo{u},\lambda^k),
\end{equation}

We consider for the variable $\bo{u}$ the partition
$\bo{u}=\left[\bo{u}_1^T \dots \bo{u}_M^T\right]^T$ and the
constraints set $\bo{U}$ can be represented in the form of a
Cartesian product $\bo{U}=\bo{U}_1\times \dots \times \bo{U}_M$,
with $\bo{u}_i \in \bo{U}_i \subseteq \rset^{n_i}$ being simple
sets, i.e. the projection on these sets can be computed very
efficiently. We also define the following partition of the identity
matrix: $I=\left[E_1 \dots E_M\right] \in \rset^{n \times n}$, where
$E_i \in \rset^{n \times n_i}$ for all $i=1, \dots, M$,
$n=\sum_{i=1}^M n_i$. Thus, $\bo{u}$ can be represented as:
$\bo{u}~=~\sum^{M}_{i=1} E_i \bo{u}_i$.

Since for each outer iteration the dual variable $\lambda^k$ is
fixed, for the simplicity of the exposition we will drop the second
argument of $\mathcal{L}$, i.e. we will use the notation
$\mathcal{L}_k(\bo{u})=\mathcal{L}(\bo{u},\lambda^k)$. We will also
denote by $\lu_k^*=\lu_k(\bu^k)$ the optimal value of
\eqref{inner_pcd}. We define the partial gradient of $\lu_k$ at
$\bo{u}$, denoted $\nabla_i \lu_k(\bo{u}) \in \rset^{n_i}$, as
$\nabla_i \lu_k(\bo{u})= E_i^T \nabla \lu_k(\bo{u})$ for all
$i=1,\dots,M$.

We consider the following assumption on the gradient of $\lu_k$:
\begin{assumption}
\label{as_lipschitz_lagr} The gradient of $\lu_k$ is coordinatewise
Lipschitz continuous with constants $L_i > 0$, i.e. for all
$i=1,\dots,M$:
\begin{equation*}
\left \| \nabla_i \lu_k (\bo{u}+{E}_i d_i)-\nabla_i \lu_k(\bo{u})
\right \| \leq L_i \left \| d_i \right \| ~ \forall \bo{u} \in
\rset^n, d_i \in \mathbb{R}^{n_i}.
\end{equation*}
\end{assumption}

We recall that $\lu_k(\bu)=F(\bu)+\iprod{\lambda_k,h(\bu)}$.
Assumption \ref{as_lipschitz_lagr} is valid for example  if $F$ has
coordinatewise Lipschitz continuous gradient and the components of
$h$ are linear or convex quadratic functions. Note also that
coordinatewise Lipschitz continuity also implies global Lipschitz
continuity on extended space $\rset^n$, with Lipschitz constant
$\sum \limits_{i=1}^M L_i$. Further, based on Assumption
\ref{as_strong_lipschitz}, since $F$ is
$\sigma_{\mathrm{F}}$-strongly convex we have that $\lu_k$ is also
strongly convex (with a parameter $\sigma_\lu$) w.r.t. the Euclidean
norm.  We also assume that $\bo{U}_i \subseteq \rset^{n_i}$ are
simple, compact, convex sets (e.g. hyperbox, Euclidean ball, entire
space $\rset^{n_i}$, etc). There exist many parallel algorithms in
the literature for solving the optimization problem
\eqref{inner_pcd}: e.g. Jacobi algorithms \cite{BerTsi:89,
DoaKev:11}, coordinate descent methods \cite{SteVen:10}, etc.
However, the rate of convergence for these algorithms is guaranteed
under more conservative assumptions than the ones required for the
parallel coordinate descent  method proposed in this section.

Due to  Assumption \ref{as_lipschitz_lagr} we have \cite[Section
2]{c2}:
\begin{equation}\label{ec7}
\lu_k(\bo{u}+E_i d_i)\!\leq\! \lu_k(\bo{u})+ \left <\nabla_i \lu_k
(\bo{u}), d_i \right
>+\frac{L_i}{2} \left \| d_i \right \|^2 ~
\forall \bo{u} \in \rset^n, \;  d_i \in \mathbb{R}^{n_i}, \;
i=1,\dots,M.
\end{equation}
We introduce the following norm for the extended space
$\mathbb{R}^n$:
\begin{equation}\label{ec12} \left\| \bo{u} \right
\|_1^2=\sum^{M}_{i=1}L_i \left \| \bo{u}_i \right \|^2,
\end{equation}
which will prove useful for estimating the rate of convergence for
our algorithm. Since $\lu_k$ is $\sigma_\lu$-strongly convex w.r.t.
the Euclidean norm, it is also strongly convex w.r.t. $\left \|
\cdot \right \|_1$
 with parameter $\sigma_1 \leq \frac{\sigma_\lu}{L_{\text{max}}}$,
  where $L_{\text{max}}=\max \limits_{i=1,\dots,M} L_i$. Then, the following inequality holds \cite{Nes:04}:
\begin{equation}\label{strong_conv}
\lu_k(\wb)\!\geq\!\lu_k(\bu) + \left < \nabla
\lu_k(\bu),\wb\!-\!\bu\right
> + \frac{\sigma_1}{2} \left \|\wb-\bu \right \|_1^2 \;\;  \forall
\bo{w}, \bo{u} \in \rset^n
\end{equation}
and combining it with \eqref{ec7} we can deduce that $\sigma_1
\leq 1$.

For solving the inner problem \eqref{inner_pcd} we propose the
following \textit{parallel coordinate descent method}, which is
similar to the algorithm from \cite{SteVen:10}, but has much simpler
iteration:
\begin{center}
\framebox{
\parbox{6.5cm}{
\begin{center}
{ Algorithm \textbf{(PCD)}$(\bu^{k,0})$ }
\end{center}
{Given $\bu^{k,0}$, for $l \geq 0$:\\
For $i=1,\dots, M$ compute in parallel}
\begin{enumerate}
\item{{$\vb_i^{k,l}=\left[\bu_i^{k,l}-\frac{1}{L_i}\nabla_i
\lu_k(\bu^{k,l})\right]_{\bo{U}_i}$} } \item{{
$\bu_i^{k,l+1}=\frac{1}{M}\vb_i^{k,l}+\frac{M-1}{M}\bu_i^{k,l}$}}.
\end{enumerate}
}}
\end{center}
\vspace{0.1cm}

From the optimality conditions for $\vb_i^{k,l}$ we get:
\begin{align}
\label{opcond}  \left< \nabla_i \lu_k (\bu^{k,l}) \!+\!  L_i
(\vb_i^{k,l} \!-\! \bu^{k,l}_i), \vb_i \!-\vb_i^{k,l}\!  \right >
\!\geq\! 0 \;\; \forall \vb_i \in \bo{U}_{i}.
\end{align}
Taking $\vb_i=\bu_i^{k,l}$ in \eqref{opcond} and combining with
\eqref{ec7}  and
convexity of $\lu_k$ we can conclude that algorithm \textbf{(PCD)}
decreases the objective function at each inner iteration $l$:
\begin{equation*}
 \lu_k(\bu^{k,l+1}) \leq \lu_k(\bu^{k,l})\;\;\;  \forall l
\geq 0.
\end{equation*}

\begin{remark}
Note that if the sets $\bo{U}_{i}$ are simple and $\lu_k$ has cheap
coordinate derivatives, then computing $\vb_i^{k,l}$ can be done
numerically very efficient. For example, in case of  hyperbox sets,
the projection on $\bo{U}_i$ can be done in
$\mathcal{O}\left(n_i\right)$ operations and if we also consider
$\lu_k$ to be quadratic, then the cost of computing $\nabla_i
\lu_k(\bu)$ is $\mathcal{O}(n \cdot n_i)$. Moreover, if its Hessian
is sparse, then the cost of computing $\nabla_i \lu_k(\bu)$ is
usually much cheaper. Thus, for quadratic problems the worst case
complexity per iteration of our method is $\mathcal{O}(n^2)$. Note
that the complexity per iteration of the Jacobi type methods from
\cite{BerTsi:89,DoaKev:11,SteVen:10} is at least $\mathcal{O}(n^2+
\sum_{i=1}^M n_i^3)$ provided that the local quadratic subproblems
are solved with an interior point solver.
\end{remark}

The following theorem provides the convergence rate of  algorithm
\textbf{(PCD)} and employs standard techniques for proving
convergence of the projected gradient method~\cite{c2,Nes:04}.
\begin{theorem}
\label{thc} Let Assumption \ref{as_lipschitz_lagr} hold and $\lu_k$
be $\sigma_1$-strongly convex w.r.t. $\left \| \cdot \right \|_1$.
Then,  the following linear rate of convergence is achieved for
algorithm \textbf{(PCD)}:
\begin{align*}
\lu_k(\bu^{k,l})\!-\!\lu_k^*  \!\leq\! \left ( 1 \!-\!
\frac{2\sigma_1}{M(1+\sigma_1)} \right )^{l} \!\! \left(
\frac{1}{2} r_u^{0} \!+\! \lu_k(\bu^{k,0}) \!- \! \lu_k^*
\right),
\end{align*}
where $r_u^{0}=\|\bu^{k,0}-\bu^k\|_1^2$.
\end{theorem}

\proof We introduce the following term: $r_u^{l}= \left \|
\bu^{k,l} \!-\!\bu^k \right \| _1 ^2= \sum_{i=1}^{M} L_i \left <
\bu^{k,l}_i \! - \!\bu^k_i, \bu^{k,l}_i\! - \!\bu^k_i \right >$,
where we recall that $\bu^k$ is the optimal solution of
\eqref{inner_pcd} and $\bu_i^k=E_i^T\bu^k$. Further, using
\eqref{opcond} and similar derivations as in \cite{c2} we can
write:
\begin{align*} r_u^{l+1}\!& \! =\! \sum_{i=1}^{M}
L_i \left \| \frac{1}{M} \vb_i^{k,l}+(1\! -
 \!\frac{1}{M})\bu^{k,l}_i\! - \! \bu^k_i \right \|^2 \nonumber \\
&\leq \!r_u^{l}\!\!-\!\frac{2}{M}\!\!\sum_{i=1}^{M}\!\!
\left(\!\frac{L_i}{2} \left \|
\vb_i^{k,l}\!\!\!-\!\bu_i^{k,l}\right \| ^2\!\!+\!\left < \nabla_i
\lu_k(\bu^{k,l}),\vb_i^{k,l}\!\!-\!\bu_i^{k,l} \right > \!+\!\left
< \nabla_i \lu_k(\bu^{k,l}),\bu^k_i\!-\!\bu^{k,l}_i \right>
\right). \nonumber
\end{align*}
By convexity of $\lu_k$ and \eqref{ec7} we obtain:
\begin{equation*}
\label{rez_cent} r_u^{l+1} \leq r_u^{l} -
2(\lu_k(\bu^{k,l+1})-\lu_k(\bu^{k,l}))+ \frac{2}{M} \left < \nabla
\lu_k(\bu^{k,l}), \bu^k-\bu^{k,l} \right
>.
\end{equation*}
If we now take $\wb= \bu^k$ and $\bu=\bu^{k,l}$ in
\eqref{strong_conv} and use the previous inequality we get:
\begin{equation}\label{rezult1} \frac{1}{2} r_u^{l+1}
+ \lu_k(\bu^{k,l+1})-\lu_k^* \leq \frac{1}{2} r_u^{l}
+\lu_k(\bu^{k,l})-\lu_k^* -\frac{1}{M} (\lu_k(\bu^{k,l})-\lu_k^*
+\frac{\sigma_1}{2} r_u^{l}).
\end{equation}
From the strong convexity of $\lu_k$ in \eqref{strong_conv} we
also get: $\lu_k(\bu^{k,l})-\lu_k^* + \frac{\sigma_1}{2} r_u^{l}
\geq \sigma_1 r_u^{l}$. We now define
$\gamma=\frac{2\sigma_1}{1+\sigma_1} \in [0,1]$ and using the
previous inequality  we obtain:
\begin{align*}
 \lu_k(\bu^{k,l})\!-\!\lu_k^*+\frac{\sigma_1}{2} r_u^{l}\! \leq
 &\gamma \left(\lu_k(\bu^{k,l})\!-\!\lu_k^*\!+\!\frac{\sigma_1}{2} r_u^{l} \right) + (1-\gamma)\sigma_1 r_u^{l}.
\end{align*}
Using this inequality in \eqref{rezult1} we get:
\begin{align*}\label{rezfin}
\frac{1}{2} r_{l+1}^2 \!+\! \lu_k(\bu^{k,l+1})-\lu_k^*  \!\leq\!
\left ( 1- \frac{\gamma}{M} \right ) \left ( \frac{1}{2} r_l^2
+\lu_k(\bu^{k,l}) -\lu_k^* \right ).
\end{align*}
Applying this inequality iteratively, we obtain for $l \geq 0$:
\begin{align*}
\frac{1}{2} r_{l}^2 \!+\! \lu_k(\bu^{k,l}) \!-\! \lu_k^* \!\leq\!
\left (1 \!-\! \frac{\gamma}{M} \right )^{l} \left ( \frac{1}{2}
r_0^2 +\lu_k(\bu^{k,0}) - \lu_k^*  \right ),
\end{align*}
and by replacing $\gamma=\frac{2\sigma_1}{1+\sigma_1} $ we obtain
the result. \qed

\noindent We can conclude from Theorem \ref{thc} that the number of
inner iterations $l_{\mathrm{in}}$ which has to be performed such
that stopping criterion \eqref{inner_crit} holds for an inner
accuracy $\ei$ is given by \cite{Nes:04}:
\begin{equation}
 \label{liniter}
 l_{\mathrm{in}} = \left\lfloor \frac{M L_\text{max}}{\sigma_\lu} \ln
 \frac{3 L_\text{max} D_\bo{U}^2}{\ei} \right \rfloor.
\end{equation}

The output of  algorithm \textbf{(PCD)} is  $\bar \bu^k = \bu^{k,
l_{\mathrm{in}}}$. To conclude, we present now the following
algorithmic framework for solving the original problem
\eqref{original_primal}:
\begin{algorithm}{\textit{Inexact dual  (fast) gradient method}}
\label{alg:A1}
\newline
\textbf{Initialization:} Choose an outer accuracy $\eo$.\\
 Compute $\ei$ and $k_{\mathrm{out}}$ as in Sections \ref{sec_dg} or \ref{sec_dfg}. \\
 Choose an initial point $\lambda^0 \in \rset^p_+$.
\newline
\textbf{Outer loop:} For $k=0, 1,\dots, k_{\mathrm{out}}$, perform:
\begin{itemize}
\item[]Step 1. \textbf{Inner loop}: For given $\lambda^k$,
choose $\bu^{k,0}  \in \bo{U}$. \\
 \hspace*{1.2cm} Compute $l_{\mathrm{in}}$ as in eq. \eqref{liniter}. \\
 \hspace*{1.2cm} For $l=0, 1, \dots, \l_{\mathrm{in}}$ apply algorithm \textbf{(PCD)} to obtain
$\bar \bu^k = \bu^{k, l_{\mathrm{in}}}$. \item[]Step 2. Compute
the approximate gradient $\nabla{\bar{d}}(\lambda_k) =
h(\bar{\bu}^k)$. \item[]Step 3. Update $\lambda^{k+1}$ as in Alg.
\textbf{(IDG)} or $(\lambda^{k+1},{\hat \lambda}^k)$
 as in Alg. \textbf{(IDFG)}.
\item[]Step 4.  Update average sequences ($\hat{\bu}^k$,
$\hat{\lambda}^k$).
\end{itemize}
\textbf{Output}: generated approximate primal-dual solutions
$(\hat{\bu}^k, \hat{\lambda}^k)$.
\newline
\end{algorithm}


\section{Distributed MPC problems for
constrained network systems} \label{sec_application_MPC} In this
section we apply the algorithms \textbf{(IDG)}, \textbf{(IDFG)} and
\textbf{(PCD)} for solving in a distributed fashion MPC problems
arising in network systems.

\subsection{MPC formulation for network systems}
\label{subsec_mpc_formulation} We consider discrete-time network
systems, which  are usually modelled  by a graph  whose nodes
represents subsystems and whose arcs indicate dynamic couplings
between these subsystems, defined by the following linear state
equations:
\begin{equation}
\label{mod3} x_i(t+1) = \sum _{j \in \mathcal {N}^{i} } A_{ij}
x_{j}(t) + B_{ij} u_j(t)  \qquad \forall i=1, \dots, M,
\end{equation}
where $M$ denotes the number of interconnected subsystems, $x_i(t)
\in \mathbb{R}^{n_{x_i}}$ and $u_i(t) \in \mathbb{R}^{n_{u_i}}$
represent the state and  the input of $i$th subsystem at time $t$,
$A_{ij} \in \mathbb{R}^{n_{x_i} \times n_{x_j}}$ and $B_{ij} \in
\mathbb{R}^{n_{x_i} \times n_{u_j}}$ and $\mathcal {N}^{i}$
denotes the neighbors of the $i$th subsystem including  $i$. In a
particular case frequently found in literature
\cite{MaeMun:11,NecSuy:08,SteVen:10} the influence between
neighboring subsystems is given only in terms of inputs:
\begin{equation}\label{simp1}
x_i(t+1) = A_{ii} x_i(t)  + \sum _{j \in \mathcal {N}^{i}} B_{ij}
u_j(t).
\end{equation}
We also impose local state and input constraints: \[ x_i(t) \in X_i,
\quad   u_i(t)  \in U_i \qquad \forall i=1,\dots, M, \;\;   t \geq
0,
\]
where $X_i \subseteq \mathbb{R}^{n_{x_i}}$ and  $U_i \subseteq
\mathbb{R}^{n_{u_i}}$ are simple convex sets. For a prediction
horizon of length $N$, we consider  quadratic stage and  final costs
for each subsystem $i$:
\[ \sum_{t=0}^{N-1} \left \| x_i(t) \right \|^2_{Q_i}  + \left \|
u_i(t) \right \|^2_{R_i} + \left \| x_i(N) \right \|^2_{P_i},
\]
where matrices $Q_i, P_i$  and $R_i$ are positive definite and
$\left \| x \right \|^2_{P}=~x^T P x$.

We  now formulate the centralized MPC problem for \eqref{mod3},
for a given initial state $x$:
\begin{align}
 & F^*(x) =  \min _{x_i(t),u_i(t)}  \sum_{i=1}^M \sum_{t=0}^{N-1}
\left \| x_i(t) \right \|^2_{Q_i}  + \left \|
u_i(t) \right \|^2_{R_i} + \left \| x_i(N) \right \|^2_{P_i}  \label{mod5} \\
& \text{s.t.:} \;\; x_i(t+1) =
\sum _{j \in \mathcal {N}^{i} } A_{ij}  x_j (t) + B_{ij} u_j(t), \; x_i(0) =x_i,  \nonumber \\
& \;\;\;\;\;\;\;\;  x_i(t) \in X_i, \; u_i(t) \in U_i, \; x_i(N) \in
X_i^{\text{f}}   \quad  \forall i=1,\dots,M, ~t=0,\dots N-1,
\nonumber
\end{align}
where $X_i^{\text{f}}$ are terminal sets  chosen under some
appropriate conditions to ensure stability of the MPC scheme (see
e.g. \cite{ScoMay:99}). For the input trajectory of subsystem $i$
and the overall input trajectory we use the notations:
\begin{align*}
\bo{u}_i& = \left[u_{i}(0)^T \dots u_{i}(N-1)^T \right]^T \in
\rset^{n_i}, \;\; \bo{u}=\left[\bo{u}_{1}^T \dots \bo{u}_{M}
^T\right]^T \in \rset^n.
\end{align*}

We assume in addition that the local constraints sets $U_i, X_i$ and
the  terminal sets $X_i^{\text{f}}$ are polyhedral for all
subsystems. An extension to general convex sets is straightforward
and we omit it here due to space limitations. By eliminating the
states from the dynamics \eqref{mod3}, problem \eqref{mod5} can be
expressed as a large-scale quadratic convex optimization problem of
the form:
\begin{align}\label{prob_princc}
F^*(x) = & \min_{\mathbf{u}_1 \in \bo{U}_{1}, \dots, \mathbf{u}_M
\in \bo{U}_{M}}
  \frac{1}{2} \bo{u}^T \bo{H} \bo{u} + (\Wb x + \wb)^T \bo{u}\\
& \text{s.t.:} \;\;\; \bo{G} \bo{u}+ \Eb x + \gb \leq 0, \nonumber
\end{align}
where $\bo{H} \in \rset^{n \times n}$ is positive definite  due to
the assumption that all $R_i$ are positive definite and the
inequalities   $\bo{G} \bo{u}+ \Eb x + \gb \leq 0$, with $\bo{G} \in
\rset^{p \times n}$, are obtained by eliminating the states from the
constraints $x_i(t) \in X_i$ and $x_i(N) \in X_i^{\text{f}}$ for all
$i$ and $t$. If the projection on the input constraints set $U_i$ is
difficult, we can also move the input constraints in the
complicating constraints $\bo{G} \bo{u}+ \Eb x + \gb \leq 0$. In
this case $\bo{U}_i = \rset^{n_i}$. Otherwise, i.e. the set $U_i$ is
simple (e.g.  hyperbox), the convex set $\bo{U}_i = \prod_{t=1}^N
U_i$. In MPC, at each time instant, given the initial state $x \in
X_N$, where $X_N \subseteq \prod_{i=1}^M X_i$ is a region  of
attraction \cite{ScoMay:99}, we need to solve the optimization
problem \eqref{mod5} or equivalently \eqref{prob_princc}.  We assume
for \eqref{prob_princc} that  for any $x \in X_N$ there exists a
``strict Slater'' vector $\tilde{\bu}$, i.e. $\tilde{\bu} \in
\bo{U}$ and $\Gb \tilde{\bu} + \Eb x +\gb <0$.

In the following sections we discuss how we can solve the MPC
problem \eqref{prob_princc} by combining the algorithms
\textbf{(IDG)}, \textbf{(IDFG)} and \textbf{(PCD)} with tightening
constraints techniques. We will derive  estimates for the  number of
iterations required for finding a suboptimal feasible solution.


\subsection{Tightening the coupling constraints}
\label{sec_tightening the coupling constraints} In many
applications, like e.g. the MPC problem discussed above, the
constraints may represent different requirements on physical
limitation of actuators, safety limits and operating conditions of
the controlled plant. Thus, ensuring the feasibility of the primal
variables, i.e. $\bo{u} \in \bo{U}$ and $\bo{G} \bo{u}+ \Eb x + \gb
\leq 0$, becomes a prerequisite. However, as we have seen in
Sections \ref{sec_dg} and \ref{sec_dfg}, dual methods can ensure
these requirements only at optimality, which is usually impossible
to attain in practice. Therefore, in our approach, instead of
solving the original problem \eqref{prob_princc}, we consider a
tightened problem (see also \cite{DoaKev:11} for a similar approach
where the tightened dual problem is solved using a subgradient
algorithm with very slow convergence rate of order
$\mathcal{O}\left(1/\sqrt{k}\right)$ and approximate solutions for
the inner problems are computed using the Jacobi algorithm
\cite{BerTsi:89}).

We introduce the following tightened problem associated with the
original problem \eqref{prob_princc}:
\begin{align}\label{prob_tight}
 F^*_{\ec}(x)=&\min_{\mathbf{u} \in \bo{U}}
   F(x,\bo{u}) \quad \left(=\frac{1}{2} \ \bo{u}^T \bo{H} \bo{u} + (\Wb x + \wb)^T \bo{u}\right)\\
& \text{s.t.:} \;\;\;\bo{G} \bo{u}+ \Eb x + \gb +\ec \eb \leq 0,
\nonumber
\end{align}
where $\eb \in \rset^p$ denotes the vector with all entries $1$ and
\begin{equation}
\label{choice_ec} 0 < \ec \leq \frac{1}{2}\min \limits_{j=1,\dots,p}
 \{-\left(\bo{G} \tilde{\bo{u}} +\Eb x +\gb\right)_j\},
\end{equation}
with $\tilde{\bo{u}}$ being a strict Slater vector for
\eqref{prob_princc}.  Note that for this choice of $\ec$, we have
that $\tilde{\bo{u}}$ is also  a strict Slater vector  for the
tightened problem \eqref{prob_tight}. Similar to Section
\ref{sec_mainsec}, for problem \eqref{prob_tight} we also denote by
$\lu_{\ec}$ the partial Lagrangian w.r.t. the complicating
constraints $\bo{G} \bo{u}+ \Eb x + \gb +\ec \eb \leq 0$ and by
$d_{\ec}$ the corresponding dual function.

In the following sections we will see how we can ensure the
feasibility, suboptimality  and stability  of the MPC scheme given
in \eqref{mod5} based on the suboptimal  input
$\hat{\bu}^{k_\mathrm{out}}$ obtained by solving the tightened
problem \eqref{prob_tight} with the newly developed algorithms
\textbf{(IDG)}/\textbf{(IDFG)} and~\textbf{(PCD)}.


\subsection{Feasibility and suboptimality of the MPC scheme}
\label{subsec_feas_stab_subopt} At each time instant  of the MPC
scheme, given the initial state $x$ in the region of attraction
$X_N$, instead of solving the optimization problem
\eqref{prob_princc} we solve the tightened problem
\eqref{prob_tight} using the algorithms \textbf{(IDG)} or
\textbf{(IDFG)} for the outer problem and algorithm \textbf{(PCD)}
for the inner problem. At each step we obtain a suboptimal  input
$\hat{\bu}^{k_\mathrm{out}}$ and according to the receding horizon
strategy we apply to the system only the first  input
$\hat{\bu}^{k_\mathrm{out}}(0)$.  However, we want that the
generated   control sequence $\hat{\bu}^{k_\mathrm{out}}$ to be
suboptimal and feasible for the original MPC problem
\eqref{prob_princc}. Thus, we first need to find a relation
between $F^*_{\ec}(x)$ and $F^*(x)$. Let us denote by
$\lambda_{\ec}^*$ an optimal Lagrange multiplier for the
inequality constraints in \eqref{prob_tight}. The following upper
bound can be established for any strict Slater vector
$\tilde{\bu}$ and dual multiplier $\tilde{\lambda} \in \rset^p_+$:
\begin{align}
\|\lambda_{\ec}^*\| &\overset{\eqref{eq_upper_norm_mult}}{\leq}
\frac{F(x,\tilde{\bu})-\min_{\bu \in \bo{U}}
F(x,\bu)+\iprod{\tilde{\lambda},\bo{G}\bu+\bo{E}x+\gb+\ec\eb}}{\min
\limits_{j=1,\dots,p}
 \{-\left(\bo{G} \tilde{\bo{u}} +\Eb x +\gb+\ec\right)_j\} }\nonumber\\
& = \frac{\left[F(x,\tilde{\bu})-\min_{\bu \in \bo{U}}
F(x,\bu)+\iprod{\tilde{\lambda},\bo{G}\bu+\bo{E}x+\gb}\right]-\iprod{\tilde{\lambda},\ec\eb}}{\min
\limits_{j=1,\dots,p}
 \{-\left(\bo{G} \tilde{\bo{u}} +\Eb x +\gb\right)_j\}-\ec}\nonumber\\
&\leq 2\mathcal{R}_{\mathrm{d}} \qquad \forall x \in X_N,
\end{align}
where in the last inequality we used \eqref{choice_ec} and the fact
that both $\tilde{\lambda}$ and $\ec$ are nonnegative.  Taking into
account that $\set{\bu : \bo{G}\bu+\bo{E}x+\gb+\ec\eb \leq 0}
\subseteq \set{\bu: \bo{G}\bu+\bo{E}x+\gb \leq 0}$ we have:
\begin{equation}
\label{eq_lower_bound_fstar} F^*_{\ec}(x) \geq F^*(x) \qquad \forall
x \in X_N.
\end{equation}
On the other hand, from  the dual formulation of the tightened
problem \eqref{prob_tight} we have:
\begin{align}
\label{eq_upper_bound_fstar} F_{\ec}^*(x)&=\min_{\bu \in \bo{U}}
F(x,\bu)+\iprod{\lambda_{\ec}^*,\bo{G}\bu+\bo{E}x+\gb+\ec\eb}
\nonumber\\
&=\min_{\bu \in \bo{U}}
F(x,\bu)+\iprod{\lambda_{\ec}^*,\bo{G}\bu+\bo{E}x+\gb}+\iprod{\lambda_{\ec}^*,\ec\eb} \\
&\leq \max_{\lambda \geq 0} \min_{\bu \in \bo{U}}
F(x,\bu)+\iprod{\lambda,\bo{G}\bu+\bo{E}x+\gb} +
\sqrt{p}\ec\|\lambda_{\ec}^*\| \leq
F^*(x)+2\sqrt{p}\mathcal{R}_{\mathrm{d}}\ec. \nonumber
\end{align}

We will further see how we can use relations
\eqref{eq_lower_bound_fstar} and \eqref{eq_upper_bound_fstar} to
recover the primal suboptimality for the original problem
\eqref{prob_princc} from the suboptimality of the tightened problem
\eqref{prob_tight}, based on the results from Section
\ref{sec_mainsec}. We now discuss the suboptimality and the
feasibility of the MPC scheme based on the algorithms \textbf{(IDG)}
and \textbf{(IDFG)}.

For the algorithm \textbf{(IDG)} we assume that the outer accuracy
$\eo$ is chosen such that:
\begin{equation*}
\eo \leq \left(\sqrt{p}+0.05\right)\mathcal{R}_{\mathrm{d}}\min
\limits_{j=1,\dots,p}
 \{-\left(\bo{G} \tilde{\bo{u}} +\Eb x +\gb\right)_j\}.
\end{equation*}
Based on the results stated in Section \ref{sec_dg} and relations
\eqref{eq_lower_bound_fstar} and \eqref{eq_upper_bound_fstar} we can
choose, for example, the following values for the number of outer
iterations $k_{\mathrm{out}}$, the inner accuracy $\ei$ and also for
the tightening parameter $\ec$:
\begin{align}
\label{kmpcg} &k_{\mathrm{out}}=\left\lfloor
\frac{10\left(2\sqrt{p}+0.1\right)L_{\mathrm{d}}\mathcal{R}_{\mathrm{d}}^2}{\eo}\right\rfloor\\
&\ei=\frac{\eo}{20\left(2\sqrt{p}+0.1\right)},~~~\ec=\frac{\eo}{\left(2\sqrt{p}+0.1\right)\mathcal{R}_{\mathrm{d}}}.
\nonumber
\end{align}
Using the previous choices for $k_{\mathrm{out}}$, $\ei$ and $\ec$
in Theorem \ref{theorem_fesa_grad} we have:
\begin{equation*}
\left\| \left[
\bo{G}\hat{\bu}^{k_\mathrm{out}}+\bo{E}x+\gb+\ec\eb\right]_+\right\|
\leq
\frac{8L_{\mathrm{d}}\mathcal{R}_{\mathrm{d}}}{k_{\mathrm{out}}+1}+2\sqrt{\frac{L_{\mathrm{d}}}{k_{\mathrm{out}}+1}\ei}
< \ec,
\end{equation*}
which implies that for all $j=1, \dots, p$, we can write:
\begin{equation*}
\left[ (\bo{G}\hat{\bu}^{k_\mathrm{out}} + \bo{E} x + \gb + \ec)_j
\right]_+ < \ec.
\end{equation*}
Since $ (\bo{G} \hat{\bu}^{k_\mathrm{out}} + \bo{E} x + \gb +\ec)_j
\leq \left[ ( \bo{G} \hat{\bu}^{k_\mathrm{out}} + \bo{E} x + \gb +
\ec)_j \right]_+$ we have that $ \hat{\bu}^{k_\mathrm{out}} \in
\bo{U}$ and $\bo{G}\hat{\bu}^{k_\mathrm{out}}+\bo{E}x+\gb  < 0$ and
thus algorithm \textbf{(IDG)} guarantees feasibility of the primal
variable $\hat{\bu}^{k_{\mathrm{out}}}$. Further, using now Theorem
\ref{theorem_primal_grad} together with \eqref{eq_lower_bound_fstar}
and \eqref{eq_upper_bound_fstar} we have that $
-\frac{\eo}{\sqrt{p}}\leq F(x,\hat{\bu}^{k_\mathrm{out}})-F^*(x)\leq
\eo$ and since $\hat{\bu}^{k_{\mathrm{out}}}$ is feasible, we get:
\begin{equation*}
0  \leq  F(x,\hat{\bu}^{k_\mathrm{out}})-F^*(x)\leq \eo
\end{equation*}
and thus  the MPC scheme based on algorithm \textbf{(IDG)} is also
$\eo$-suboptimal.

In order to prove the suboptimality and feasibility of the MPC
scheme based on algorithm \textbf{(IDFG)} we proceed in a similar
way as for algorithm \textbf{(IDG)}. We assume that the outer
accuracy $\eo$ is chosen such that:
\begin{equation*}
\eo \leq \left(\sqrt{p}+0.5\right)\mathcal{R}_{\mathrm{d}}\min
\limits_{j=1,\dots,p}
 \{-\left(\bo{G} \tilde{\bo{u}} +\Eb x +\gb\right)_j\}.
\end{equation*}
Based on the convergence properties of algorithm \textbf{(IDFG)}
presented in Section \ref{sec_dfg} and  relations
\eqref{eq_lower_bound_fstar} and \eqref{eq_upper_bound_fstar} we can
choose:
\begin{align}
\label{kmpcfg} &k_{\mathrm{out}}=\left\lfloor
8\sqrt{\frac{\left(2\sqrt{p}+1\right)L_{\mathrm{d}}\mathcal{R}_{\mathrm{d}}^2}{\eo}}\right\rfloor\\
&\ei=\frac{\eo\sqrt{\eo}}{8\sqrt{2}\sqrt{L_{\mathrm{d}}}\mathcal{R}_{\mathrm{d}}\left(2\sqrt{p}+1\right)^{\frac{3}{2}}},
~~~\ec=\frac{\eo}{\left(2\sqrt{p}+1\right)\mathcal{R}_{\mathrm{d}}}.
\nonumber
\end{align}
The $\eo$-suboptimality and feasibility of the MPC scheme based on
algorithm \textbf{(IDFG)} can be proved now in a similar way as the
one for algorithm \textbf{(IDG)} using Theorems
\ref{theorem_primal_fesa} and \ref{theorem_primal_optim}, i.e.:
\begin{align*}
& 0  \leq
F(x,\hat{\bu}^{k_\mathrm{out}})-F^*(x)\leq \eo~~\mathrm{and}\\
&  \hat{\bu}^{k_\mathrm{out}} \in \bo{U}, \quad
\bo{G}\hat{\bu}^{k_\mathrm{out}}+\bo{E}x+\gb < 0.
\end{align*}
In conclusion, in our MPC scheme from our suboptimal and feasible
control sequence $\hat{\bu}^{k_\mathrm{out}}$ only the first input
$\hat{\bu}^{k_\mathrm{out}}(0)$ is applied to the system according
to the receding horizon strategy.


\subsection{Stability of the MPC scheme}
\label{subsec_stab}

For stability analysis, we express for the entire network system the
dynamics, the matrices corresponding to the total stage and final
costs, and the total  terminal set as: $x(t+1) = A x(t) + B u(t), Q,
R, P$ and $X^\mathrm{f}$, respectively. Further,  the next state in
our MPC scheme is denoted  $x^+=Ax+B \hat{\bu}^{k_\mathrm{out}}(0)$
and a new sequence of feasible inputs for the MPC problem at the
next state $x^+$ is denoted with $\tilde \bu^+=\left[
\left(\hat{\bu}^{k_\mathrm{out}}(1)\right)^T \dots
\left(\hat{\bu}^{k_\mathrm{out}}(N-1)\right)^T
\left(Kx(N)\right)^T\right]^T$, where $u= K x$ is a linear feedback
controller. In this section we will make use of the following
assumptions:
\begin{assumption}
\label{as_invariant} (i) The terminal constraint set $X^\mathrm{f}$
is positively invariant for the closed-loop system $x(t+1)= (A + B
K)x(t)$,  i.e.  for all $x \in \mathrm{int}(X^\mathrm{f})$ we have
that $(A+BK)x \in \mathrm{int}(X^\mathrm{f})$.

(ii) The following relation holds:
\begin{equation}
\label{eq_lyap} F(x^{+},\tilde \bu^+)\leq
F(x,\hat{\bu}^{k_\mathrm{out}}) - \|x\|_Q^2 \quad \forall x \in X_N.
\end{equation}
\end{assumption}
Assumption \ref{as_invariant} is standard in the  the MPC framework
(see also \cite{ScoMay:99, DoaKev:11}). Moreover, distributed
synthesis procedures for finding the matrices $K$ an $P$ for the
terminal controller and terminal cost such that  Assumption
\ref{as_invariant} holds can be found e.g. in \cite{NecCli:12}.

Based on Assumption \ref{as_invariant} (i) and the fact that
$\bo{G}\hat{\bu}^{k_\mathrm{out}}+\bo{E}x+\gb < 0$ we can
immediately see that $\tilde \bu^+$ is a strict Slater vector of the
MPC problem \eqref{prob_princc} with initial state $x^+$. Therefore,
in the MPC problem for the next state $x^+$ we update the strict
Slater vector as explained above, i.e.: \[ \tilde \bu^+=\left[
\left(\hat{\bu}^{k_\mathrm{out}}(1)\right)^T \dots
\left(\hat{\bu}^{k_\mathrm{out}}(N\!-\!1)\right)^T
\left(Kx(N)\right)^T\right]^T \] and thus  $\tilde \bu^+$ is also
feasible for  tightened problem~\eqref{prob_tight}.

In order to prove asymptotic stability of the MPC scheme for all $x
\in X_N$ we use similar arguments as in \cite{ScoMay:99, DoaKev:11}
by   showing that $F(x,\hat{\bu}^{k_\mathrm{out}})$ is a Lyapunov
function:
\begin{align*}
F(x^+, (\hat{\bu}^{k_\mathrm{out}})^+) &\leq F^*(x^+) + \eo^+  \leq
F_{\ec}^*(x^+) + \eo^+ \leq F(x^+, \tilde \bu^+) + \eo^+ \\
&\overset{\eqref{eq_lyap}}{ \leq } F(x,\hat{\bu}^{k_\mathrm{out}}) -
\|x\|_{Q}^2 + \eo^+,
\end{align*}
where $\eo^+$ denotes the outer accuracy for solving MPC problem
\eqref{prob_tight} at initial state $x^+$.  From the previous
discussion we have that choosing e.g.
\begin{align}
\label{eoutp} \eo^+ \leq \min \left\{ \frac{1}{2} \|x\|_{Q}^2, \;
c(p) \min \limits_{j=1,\dots,p}
 \{-\left(\bo{G} \tilde \bu^+ +\Eb x^+ +\gb\right)_j\}
\right\},
\end{align}
 we get asymptotic stability of  the closed-loop  system.
Here, $c(p) = \sqrt{p}+0.05$ for algorithm \textbf{(IDG)} and
$c(p)=\sqrt{p}+0.5$ for \textbf{(IDFG)}.


\subsection{Distributed implementation}
\label{subsec_distributed_implementation}

\noindent In this section we discuss some technical aspects for the
distributed implementation of our inexact dual decomposition methods
in the case of MPC problem \eqref{mod5} and its equivalent
form~\eqref{prob_princc}.

 Usually, for the dynamics
\eqref{mod3} the corresponding matrices $\bo{H}$ and $\Gb$ obtained
after eliminating the states  are dense and despite the fact that
algorithms \textbf{(IDG)}, \textbf{(IDFG)} and \textbf{(PCD)} can
perform parallel computations (i.e. each subsystem needs to solve
small local problems) we need communication between $N$ steps
neighborhood subsystems \cite{CamSch:11,DoaKev:11}. However, for the
dynamics \eqref{simp1} the corresponding matrices $\bo{H}$ and $\Gb$
are sparse and in this case in our algorithms \textbf{(IDG)},
\textbf{(IDFG)} and \textbf{(PCD)} we can perform distributed
computations (i.e. the subsystems solve small local problems in
parallel and they need to communicate only with one neighborhood
subsystems as detailed below). Indeed, if the dynamics of the
subsystems are given by \eqref{simp1}, then $x_i(t) = A_{ii}^t
x_i(0) + \sum_{l=1}^t \sum _{j \in \mathcal {N}^{i}} A_{ii}^{l-1}
B_{ij} u_j(t-l)$ and thus the matrices $\bo{H}$ and $\bo{G}$ have a
sparse structure (see e.g. \cite{CamSch:11,SteVen:10}). In
particular,  the complicating constraints have the following
structure: for matrix $\Gb$ the $(i,j)$ block matrices of $\Gb$,
denoted $\Gb_{ij}$, are zero for all $j \notin \mathcal {N}^{i}$ for
a given subsystem $i$, while the matrix $\Eb$ is block diagonal.
Further, if we define the neighborhood subsystems of a certain
subsystem $i$ as $ \hat {\mathcal N}^i = {\mathcal N}^i \cup \{l: \;
l \in {\mathcal N}^j, j \in \bar {\mathcal N}^i\}$, where $\bar
{\mathcal N}^i = \{ j: \; i \in  {\mathcal N}^j \}$, then the matrix
$\bo{H}$ has all the block matrices $\bo{H}_{ij} = 0$ for all $j
\notin  \hat {\mathcal N}^i$ and the matrix $\Wb$ has all the block
matrices $\Wb_{ij} = 0$ for all $j \notin  \bar {\mathcal N}^i$, for
any given subsystem $i$. Thus, the $i$th block components of both
$\bar{\nabla} d_{\ec}$ and $\nabla L_{\ec}(\bo{u},\lambda)$ can be
computed using only local information, i.e. each subsystem
$i=1,\dots,M$ does the following synchronous computations:
\begin{align}
& \bar{\nabla}_i d_{\ec}(\lambda)= \sum_{j \in \mathcal {N}^{i}}
\Gb_{ij} {\bo{u}}_j + \Eb_{ii} x_i + \gb_i + \ec \eb \label{grad_d}\\
 &  \nabla_i L_{\ec}(\bo{u},\lambda) = \sum_{j \in \hat {\mathcal N}^{i}}
\bo{H}_{ij} \bo{u}_j +  \sum_{j \in \bar {\mathcal N}^{i}} \left(
\bo{W}_{ij} x_j + \Gb_{ji}^T \lambda_j \right) + \bo{w}_i.
\label{grad_L}
\end{align}
Note that in the algorithm \textbf{(PCD)} the only parameters that
we need to compute are the  Lipschitz constants $L_i$. However, in
the MPC problems, $L_i$ does not depend on the initial state $x$ and
can be computed once, offline, locally by each subsystem $i$ as:
$L_i = \lambda_{\max}(\bo{H}_{ii})$. From the previous discussion it
follows immediately that each subsystem $i$ performs  the inner
iterations of algorithm \textbf{(PCD)}  in parallel using
distributed computations (see \eqref{grad_L}) for all $x \in X_N$.

Since the algorithms \textbf{(IDG)} and \textbf{(IDFG)} use only
first order information, we can observe that once $\bar{\nabla}_i
d_{\ec}(\lambda)$ has been computed distributively, as proved in
\eqref{grad_d}, all the computations for updating the block
component corresponding to subsystem  $i$ in $\lambda^k$ or ${\hat
\lambda}^k$ can be done in parallel due to the fact that we have to
do only vector operations. However, in these schemes all subsystems
need to know the global Lipschitz constant $L_{\text{d}} =
\frac{\|\bo{G}\|^2}{\lambda_{\text{min}}(\bo{H})}$ that usually is
difficult to be computed distributively. In practice, a good upper
bound on $L_{\text{d}}$ is sufficient, e.g. $ L_{\text{d}} \leq
\frac{\|\bo{G}\|^2_F}{\min_i \lambda_{\text{min}}(R_i)}$, where
recall that $\| \cdot \|_F$ denotes the Frobenius norm. Note that
$L_{\text{d}}$ does not depend on $x$ and can be computed offline,
before starting the MPC scheme.

In both algorithms \textbf{(IDG)} and \textbf{(IDFG)}, another
global constant that has to be updated is the upper bound on the
norm of the optimal multiplier, $\mathcal{R}_{\mathrm{d}}$. Based on
the theory developed in the previous sections, after some long but
straightforward computations an easily computed upper bound for the
next $\mathcal{R}_{\mathrm{d}}^+$ corresponding to the MPC problem
with initial state $x^+$ is given by:
\begin{align}
\label{Rd}
\mathcal{R}_{\mathrm{d}}^+ \leq \frac{\ec \langle
\hat{\lambda}^{k_\mathrm{out}}, e \rangle + 4 \eo - \|x\|_Q^2 }{\min
\limits_{j=1,\dots,M} \{ - (\bo{G} \tilde \bu^+ + \bo{E} x^+ + g)_j
\}}.
\end{align}
 Note that these upper bounds on $ L_{\text{d}}$ and
$\mathcal{R}_{\mathrm{d}}^+$ can be computed distributively in an
efficient way.

From the previous discussion we can conclude that the sequences
 $\lambda_k$, ${\hat \lambda}_k$ and $\mathbf{ \bar u}_k$, generated by the
algorithms \textbf{(IDG)}/\textbf{(IDFG)} and \textbf{(PCD)} can
be computed in parallel and distributively provided that good
estimates for $L_{\text{d}}$ and $\mathcal{R}_{\mathrm{d}}$ are
known by each subsystem. The effects of the upper bound for
$\mathcal{R}_{\mathrm{d}}$ on the overall performance of the MPC
scheme  are discussed in Sections \ref{test2}.

\section{Numerical tests}
\label{sec_numerical}

In order to certify the efficiency of the proposed algorithms, we
consider different numerical scenarios. We first analyze the
behavior of algorithms \textbf{(IDG)}, \textbf{(IDFG)} and
\textbf{(PCD)} on randomly generated QP problems and then we
compare our algorithms with other QP solvers used in the context
of distributed MPC.  The algorithms were implemented on a PC, with
2 Intel Xeon E5310 CPUs at 1.60 GHz and 4Gb of RAM.

\subsection{Practical behavior of newly developed algorithms
\textbf{(IDG)}, \textbf{(IDFG) and \textbf{(PCD)}} }
\label{test1}

We consider random QP problems of the form:
\begin{equation}
\label{testqp} F^*=\min \limits_{\mathrm{lb} \leq \bu \leq
\mathrm{ub}, G \bu +g \leq 0} F(\bu)  \quad (=0.5 \bu^T H \bu + w^T
\bu),
\end{equation}
where matrices $H \in \rset^{n \times n}$ and $G \in \rset^{2n
\times n}$ are taken from a normal distribution with zero mean and
unit variance. Matrix $H$ is then made positive definite by the
transformation $H \leftarrow H^T H + I_n$. Further, $\mathrm{ub} = -
\mathrm{lb} = 1$ and $w, g$ are taken from a uniform distribution.
For different QP dimensions ranging from $n = 100$ to $n = 1000$, we
first  analyze the behavior of algorithms \textbf{(IDG)} and
\textbf{(IDFG)} in terms of the parameters choice.

For each $n$, we consider two different estimates for the number of
outer iterations depending on the way we compute  $R_\mathrm{d} = \|
\lambda^*\|$, where
 $\lambda^*$ is an  optimal Lagrange multiplier. For algorithm
\textbf{(IDG)}, $k_{\mathrm{out}}^G$ is the average number of
iterations obtained using the bound $\mathcal{R}_\mathrm{d}$ given
in  \eqref{eq_upper_norm_mult} - Section \ref{sec_dg}, while
$k_{\mathrm{out,samp}}^G$ is the average number of iterations
obtained with $R_\mathrm{d} = \| \lambda^*\|$, where $\lambda^*$ is
computed exactly using Matlab's \texttt{Quadprog}, iterations which
correspond to $10$ random  QP problems. We also compute the average
number of outer iterations $k_{\mathrm{out,real}}^G$ observed in
practice, obtained by imposing the stopping criteria $|F(\hat
\bu^{k_{\mathrm{out,real}}^G}) - F^{*}|$ and $\|\left[G \hat
\bu^{k_{\mathrm{out,real}}^G} +g\right]_+\|$  to be less than the
estimates established in Section \ref{sec_dg} for an outer accuracy
$\eo = 10^{-3}$. Using the results from Section \ref{sec_dfg} we
compute in a similar way $k_{\mathrm{out}}^{FG}$,
$k_{\mathrm{out,samp}}^{FG}$ and $k_{\mathrm{out,real}}^{FG}$ for
algorithm \textbf{(IDFG)}. The results for both algorithms are
presented in Figure \ref{fig:outer_iterations_random_qp}.
\begin{figure}[h!]
\centering\includegraphics[angle=0,height=5cm,width=14cm]{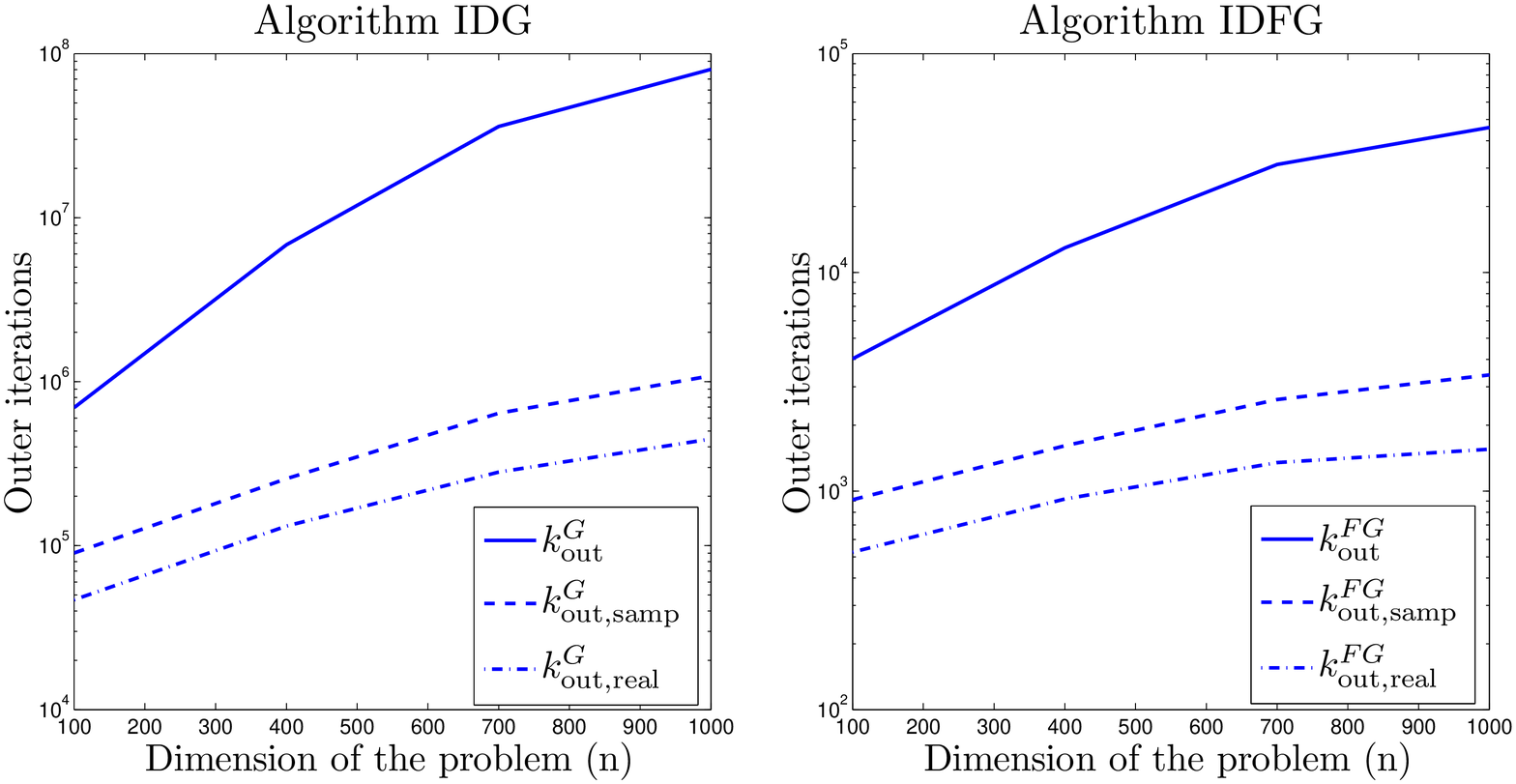}
\caption{Values of $k_{\mathrm{out}}^s$, $k_{\mathrm{out,samp}}^s$
and $k_{\mathrm{out,real}}^s$ ($s=\set{G;FG}$) for algorithms
\textbf{(IDG)} (left) and  \textbf{(IDFG)} (right),
$\eo=10^{-3}$.} \label{fig:outer_iterations_random_qp}
\end{figure}
We can observe that in practice algorithm \textbf{(IDFG)} performs
much better than algorithm \textbf{(IDG)}. Note that the expected
number of outer iterations
 $k_{\mathrm{out,samp}}^G$ and $k_{\mathrm{out,samp}}^{FG}$ obtained
from our derived bounds in Sections \ref{sec_dg} and \ref{sec_dfg}
offer a good approximation for the real number of iterations of the
two algorithms. Thus, these simulations show that our derived bounds
are tight. But, when in our derived estimates we use
$\mathcal{R}_\mathrm{d} $, then $k_{\mathrm{out}}^{FG}$ is about one
order of magnitude, while $k_{\mathrm{out}}^G$ is about two orders
of magnitude greater than the real number of iterations.

\begin{figure}[h!]
\vspace{-0.1cm}
\centering\includegraphics[angle=0,height=7cm,width=14cm]{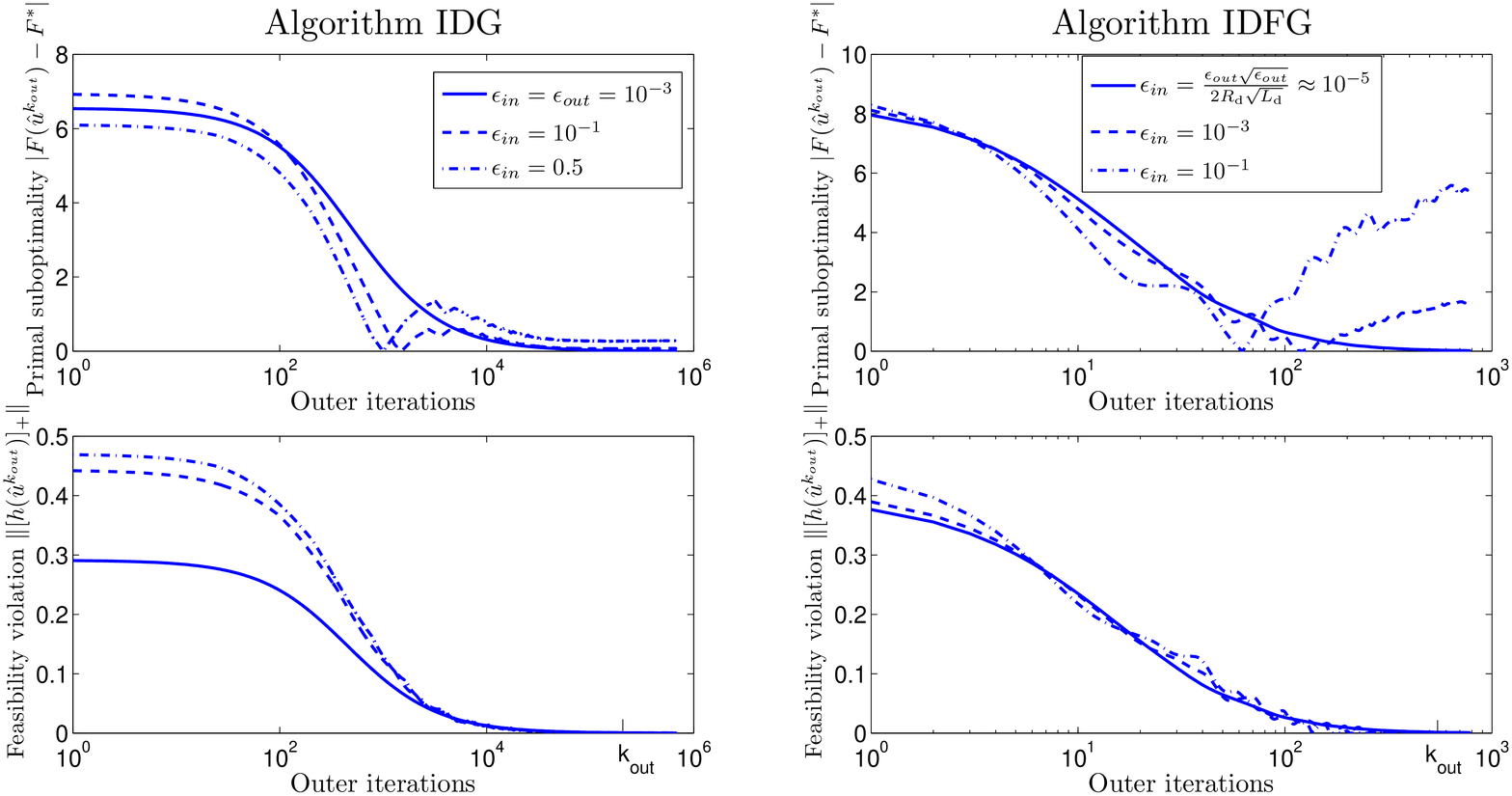}
\caption{Suboptimality and feasibility violation for algorithms
\textbf{(IDG)} (left) and  \textbf{(IDFG)} (right) for
$\eo=10^{-3}$ and different~$\ei$.}
\label{fig:inner_influence_random_qp}
\end{figure}

Since the estimates for suboptimality and feasibility violation are
also dependent on the way the inner accuracy $\ei$ is chosen, we are
also interested in the behavior of the two algorithms w.r.t. $\ei$.
For this purpose, we apply algorithms \textbf{(IDG)} and
\textbf{(IDFG)} for solving a random QP problem of dimension
$n=300$, with a fixed outer accuracy $\eo = 10^{-3}$ and different
values of $\ei$. In Figure \ref{fig:inner_influence_random_qp} we
plot the primal suboptimality and the feasibility violation by
letting the two algorithms perform the number of outer iterations
computed in Sections \ref{sec_dg} and \ref{sec_dfg}. We can observe
from Figure \ref{fig:inner_influence_random_qp} that if the inner
accuracy $\ei$ is chosen too large, the desired level of
suboptimality cannot be attained. We can also see that algorithm
\textbf{(IDG)} is less sensitive to the choice of the inner accuracy
$\ei$ than algorithm \textbf{(IDFG)} due to the fact that algorithm
\textbf{(IDFG)} accumulates errors (see Theorems
\ref{theorem_primal_grad} and \ref{theorem_primal_optim}).

In conclusion, we notice from the results of Sections \ref{sec_dg}
and \ref{sec_dfg} and  simulations that there is a tradeoff between
the speed of convergence and robustness: e.g. algorithm
\textbf{(IDFG)} is  faster than algorithm \textbf{(IDG)}, but the
second one is more robust since it does not accumulate the errors.
Thus, depending on the application, one can choose between the two
algorithms.

\begin{table}[h]
\centering \scriptsize
\begin{tabular}{|c|c|c|c|c|c|c|c|} \hline
 \multicolumn{2}{|c|}{}  & \multicolumn{2}{|c|}{ \textbf{(PCD)} } & \multicolumn{2}{|c|}{Jacobi } & \textbf{(PCD)}  & Quadprog \\
  \multicolumn{2}{|c|}{}  &     \multicolumn{2}{|c|}{}   &
  \multicolumn{2}{|c|}{\cite{SteVen:10}}      & centralized     &          \\ \hline
 M & n & CPU (sec) & Iter & CPU (sec) & Iter & CPU (sec) & CPU (sec)  \\ \hline
 &100 &  0.09  & 262 & 0.38 & 82 & 0.16 & 0.08\\
10  & 500 & 0.61 & 1244 & 2.12 & 715 & 0.75 & 1.27 \\
  & 800 & 2.11 & 2702 & 19.3 & 1274 & 9.3 & 3.8 \\
 & 1000 & 2.69  & 2851   &   23.05    &  1375     & 10.1   &  4.1   \\ \hline
\end{tabular}
\caption{ CPU  time (seconds) and number of inner iterations for
algorithms \textbf{(PCD)} and Jacobi \cite{SteVen:10}.  }
\label{numerical_table}
\end{table}

We also implemented for comparison, algorithms \textbf{(PCD)} and
Jacobi from \cite{SteVen:10}. Both algorithms were implemented in C
code, with parallelization ensured via MPI.  Table
\ref{numerical_table} presents the average CPU time in seconds and
number of iterations for  each algorithm for $10$ random QP problems
\eqref{testqp} with only box constraints.   Since the convergence
rate for algorithm in \cite{SteVen:10}  is not known, the stopping
criterion for each algorithm is $F(\bu^k) - F^* \leq 10^{-3}$, with
$F^*$ being precomputed  using \texttt{Quadprog}. As we can see
algorithm \textbf{(PCD)} is about $10$ times faster than the
algorithm~in~\cite{SteVen:10}.  


\subsection{ MPC for traffic networks}
\label{test2}

In this section we analyze the behavior of algorithms \textbf{(IDG)}
and  \textbf{(IDFG)}  on MPC problems for traffic network systems.
In \cite{CamOli:09} the authors show that traffic network systems
can be modeled in the form \eqref{simp1}. We generated ring traffic
networks with $M$ even junctions (subsystems) and having $M/2$ input
links and $M/2$ output links distributed randomly. In order to work
with small costs, we normalized the state of the system as: $x
\leftarrow x/10^3$. For the parameters of the system and of the MPC
problem see \cite{CamOli:09} and the references therein. Note that
the number of states or inputs in this traffic network  is $3M/2$.

\begin{table}[h!]
\centering \scriptsize
\begin{tabular}{|c|c|c|c|}
\hline Avg. no. of iter. &  $M=6$   &   $M=12$   &    $M =18$ \\
\hline
$k_{\mathrm{out}}^{FG}$  & $194$ & $327$ & $443$\\
\hline

$k_{\mathrm{out}}^G$  & $9452$ & $26734$ & $49113$\\
\hline

$k_{\mathrm{out}}^{SG}$  & $4.9 \cdot 10^{5}$ & $9.9 \cdot 10^{5}$ & $1.5 \cdot 10^{6}$\\
\hline
\hline

$k_{\mathrm{out,real}}^{FG}/F(x,{\hat
\bu}^{k_{\mathrm{out,real}}^{FG}})-F^*(x)$ & $57/1.3 \cdot 10^{-4} $  &  $71/2.7 \cdot 10^{-4}$ & $89/3.8 \cdot 10^{-4}$\\
\hline

$k_{\mathrm{out,real}}^G/F(x,{\hat
\bu}^{k_{\mathrm{out,real}}^G})-F^*(x) $ & $726/1.8 \cdot 10^{-4}$ &
$1289/2.5 \cdot 10^{-4}$  & $1836/3.4 \cdot 10^{-4}$
\\ \hline

$k_{\mathrm{out,real}}^{SG}$ &  $2\cdot 10^4$ &  $2\cdot 10^4$ &  $2\cdot 10^4$   \\

$|F(x,{\hat \bu}^{k_{\mathrm{out,real}}^{SG}})-F^*(x)|$  &  $9.8
\cdot
10^{-3}$  &  $5.8 \cdot 10^{-2}$ & $8.2 \cdot 10^{-2}$ \\

$\max_j \{(\Gb {\hat \bu}^{k_{\mathrm{out,real}}^{SG}} + \Eb x +
\gb)_j
\}$ & $1.2 \cdot 10^{-4}$ & $5.6 \cdot 10^{-4}$ &  $2.4 \cdot 10^{-3}$ \\
\hline
\end{tabular}
\caption{Averaged number of iterations and cost decrease for
different number of junctions (subsystems).} \label{table_traffic}
\end{table}

The distributed  MPC approach with a prediction horizon of $N=10$
steps  was applied  for solving  a single time  step of   the
traffic network with  $M \in \{6, 12, 18\}$  number of junctions
using  the newly developed algorithms \textbf{(IDG)} and
\textbf{(IDFG)} and the  dual subgradient algorithm  in
\cite{DoaKev:11}.  For each $M$ the results are shown for a set of
$10$ initial states obtained at random. Additionally to the input
constraints considered in \cite{CamOli:09} we also assume box
constraints on the states. We solve the tightened problem
\eqref{prob_tight}, obtained from the MPC problem of form
\eqref{mod5} or equivalently \eqref{prob_princc}, with an outer
accuracy $\eo=10^{-2}$.  In Table \ref{table_traffic} we report the
average number of outer iterations $k_{\mathrm{out}}^{FG}$,
$k_{\mathrm{out}}^G$ and $k_{\mathrm{out}}^{SG}$ performed by the
algorithms \textbf{(IDFG)}, \textbf{(IDG)} and the algorithm in
\cite{DoaKev:11}, respectively.
  We also count the average real number of iterations performed by
algorithms  \textbf{(IDFG)} and \textbf{(IDG)} by imposing the
stopping criterion $F(x,{\hat \bu}^{k_{\mathrm{out}}})-F^* \leq \eo$
and $\Gb {\hat \bu}^{k_{\mathrm{out}}} +\Eb x +\bo{g} \leq 0$. For
the dual subgradient algorithm in \cite{DoaKev:11} the stopping
criterion was chosen as follows: $F(x,{\hat
\bu}^{k_{\mathrm{out}}})-F^* \leq \eo^{SG}$ and $\Gb {\hat
\bu}^{k_{\mathrm{out}}} +\Eb x +\bo{g} \leq 0$, where $\eo^{SG}$ and
the rest of the parameters for this algorithm are computed  as in
\cite[Section III.C]{DoaKev:11}. In all three algorithms the inner
problems were solved with algorithm \textbf{(PCD)}. From Table
\ref{table_traffic} we observe that algorithm  \textbf{(IDFG)} has
the best behavior compared to \textbf{(IDG)} and the dual
subgradient  algorithm in \cite{DoaKev:11}. Thus, algorithm
\textbf{(IDFG)} is superior in terms of both, predicted
(theoretical) and real number of iterations (e.g. from $10$ to $100$
times faster than \textbf{(IDG)}). Further, algorithm \textbf{(IDG)}
is able to produce a feasible and suboptimal solution in a
reasonable number of
 outer iterations, while the dual subgradient algorithm in \cite{DoaKev:11}
failed to generate a feasible solution within $2 \cdot 10^4$ outer
iterations. We observed that this behavior   is due mainly to the
fact that the step size in \textbf{(IDG)} is larger than that in
\cite{DoaKev:11}.

\section{Conclusions}
\label{sec_conclusion} Motivated by  MPC problems for complex
interconnected systems, we have proposed two dual based methods for
solving large-scale smooth convex optimization problems with
coupling constraints. We  moved the coupling constraints into the
cost using duality theory. We  solved the inner subproblems only up
to a certain accuracy by means of a parallel coordinate descent
method for which we have proved linear convergence. For solving the
outer problems, we  developed inexact dual gradient and fast
gradient schemes for which we provide a full convergence analysis,
deriving upper bounds on dual and primal suboptimality and primal
feasibility violation. We also discussed some implementation issues
of the new algorithms for distributed MPC problems and tested them
on several practical applications.


\section*{Acknowledgment}
The research leading to these results has received funding from: the
European Union (FP7/2007--2013) under grant agreement no 248940;
CNCS (project TE-231, 19/11.08.2010); ANCS (project PN II,
80EU/2010); POSDRU/89/1.5/S/62557 and POSDRU/107/1.5/S/76909. \\
The authors thank Y. Nesterov,  D. Doan and T. Keviczky for
interesting discussions.

\section*{Appendix}

{\em Proof of Lemma \ref{technical_Lipschitz}}.\\
{\em Case 1} -  We first consider the unconstrained case, i.e.
$\bo{U}=\rset^n$. Since $F$ is strongly convex, it follows that
$\bo{u}(\lambda)$ is unique and thus $d$ is a differentiable
function having the gradient:
\begin{align*}
\nabla d(\lambda)&=\nabla \bo{u}(\lambda)^T\nabla
F(\bo{u}(\lambda))+h(\bo{u}(\lambda))+\nabla
\bo{u}(\lambda)^T\nabla h(\bo{u}(\lambda))^T \lambda\\
&=\nabla \bo{u}(\lambda)^T\left[\nabla F(\bo{u}(\lambda))+\nabla
h(\bo{u}(\lambda))^T
\lambda\right]+h(\bo{u}(\lambda))=h(\bo{u}(\lambda)),
\end{align*}
where the last equality is obtained using the optimality conditions
for $\bo{u}(\lambda)$, i.e.:
\begin{equation}
\label{eq_optimality_lipschitz} \nabla F(\bo{u}(\lambda))+ \nabla
h(\bo{u}(\lambda))^T \lambda=0 ~~\forall \lambda \geq 0.
\end{equation}

Taking now into account that the components of $h$ are twice
differentiable we have:
\begin{equation}
\label{eq_hess_dual} \nabla^2 d(\lambda)=\nabla
h(\bo{u}(\lambda))\nabla \bo{u}(\lambda).
\end{equation}
Differentiating now the optimality conditions
\eqref{eq_optimality_lipschitz} w.r.t. to $\lambda$ we can write:
\begin{align*}
\nabla \bo{u}(\lambda)^T \nabla^2 F(\bo{u}(\lambda))+\nabla
h(\bo{u}(\lambda))+\nabla \bo{u}(\lambda)^T \sum_{i=1}^p \lambda_i
\nabla^2 h_i(\bo{u}(\lambda))=0,
\end{align*}
from which we obtain:
\begin{equation*}
\nabla \bo{u}(\lambda)^T=-\nabla h(\bo{u}(\lambda))\left[\nabla^2
F(\bo{u}(\lambda))+\sum_{i=1}^p \lambda_i \nabla^2
h_i(\bo{u}(\lambda))\right]^{-1}.
\end{equation*}
Introducing this relation into \eqref{eq_hess_dual} and taking into
account that $\sum_{i=1}^p \lambda_i \nabla^2
h_i(\bo{u}(\lambda))\succeq 0$ we have:
\begin{align*}
-\nabla^2 d(\lambda)&=\nabla h(\bo{u}(\lambda))\left[\nabla^2
F(\bo{u}(\lambda))+\sum_{i=1}^p \lambda_i \nabla^2
h_i(\bo{u}(\lambda))\right]^{-1}\nabla h(\bo{u}(\lambda))^T\\
& \preceq \nabla h(\bo{u}(\lambda))\left[\nabla^2
F(\bo{u}(\lambda))\right]^{-1}\nabla h(\bo{u}(\lambda))^T.
\end{align*}
Since $F$ is $\sigma_{\mathrm{F}}$-strongly convex and thus
$\nabla^2 F(\bo{u}(\lambda)) \succeq \sigma_{\mathrm{F}}I_n$ and the
Jacobian of $h$ is bounded (see Assumption
\eqref{as_strong_lipschitz}), we can write further:
\begin{equation*}
\|\nabla^2 d(\lambda)\| \!\leq\! \|\left[\nabla^2
F(\bo{u}(\lambda))\right]^{-1}\!\| \cdot  \|\nabla
h(\bo{u}(\lambda))\|^2  \!\leq\! \|\left[\nabla^2
F(\bo{u}(\lambda))\right]^{-1}\!\| \cdot  \|\nabla
h(\bo{u}(\lambda))\|^2_F  \leq\!\!
\frac{c_{\mathrm{h}}^2}{\sigma_{\mathrm{F}}}.
\end{equation*}
 Thus, we can conclude using Lemma 1.2.2
from \cite{Nes:04} that
$L_{\mathrm{d}}=\frac{c_{\mathrm{h}}^2}{\sigma_{\mathrm{F}}}$.

{\em Case 2} - We assume now that $\bo{U}$ is a compact convex set.
Since $F$ is strongly convex, the dual function $d$ is still
differentiable and  given by $\nabla d(\lambda)=h(\bo{u}(\lambda))$.
In order to show Lipschitz continuity of the gradient, we consider
the following family of dual functions $(d_\tau)_{\tau > 0}$:
\begin{equation}
\label{bar} d_{\tau}(\lambda)=\min_{\bo{u}\in \rset^n}
F(\bo{u})+\iprod{\lambda,h(\bo{u})}+\tau b_{\bo{U}}(\bo{u}),
\end{equation}
where $b_{\bo{U}}$ is a self-concordant barrier function for the set
$\bo{U}$. Let $u(\lambda,\tau)$ be the optimal solution of
\eqref{bar}. Using the same reasoning as in the unconstrained case
and taking into account that $\nabla^2 b_{\bo{U}}(\bo{u}) \succeq 0$
(see \cite[Section 4.2.2]{Nes:04}), we have that for any given $\tau
> 0$ the gradient $\nabla d_{\tau}(\lambda)=h(\bo{u}(\lambda, \tau))$
is Lipschitz continuous with constant
$L_{\mathrm{d_{\tau}}}=\frac{c_{\mathrm{h}}^2}{\sigma_{\mathrm{F}}}$,
i.e. $\| h(u(\lambda,\tau)) - h(u(\nu,\tau)) \| \leq
\frac{c_{\mathrm{h}}^2}{\sigma_{\mathrm{F}}} \| \lambda - \nu \|$
for all $\lambda, \nu \geq 0$. Since for all $\lambda \geq 0$ we
have $d_{\tau}(\lambda) \to d(\lambda)$, $u(\lambda,\tau) \to
u(\lambda)$  as $\tau \to +0$ and $h$ is a continuous function we
can conclude that the gradient of the dual function $d$  is also
Lipschitz continuous with constant
$L_{\text{d}}=\frac{c_\text{h}^2}{\sigma_{\text{F}}}$. \qed

\end{document}